\documentclass[a4paper,11pt]{article}
\usepackage[latin1]{inputenc}
\usepackage[hmargin={25mm,25mm},vmargin={25mm,30mm}]{geometry}
\usepackage{amsmath,amsfonts,amsthm,amssymb}
\usepackage{cancel}
\usepackage{comment}
\usepackage{enumerate}
\usepackage{graphicx}
\usepackage[ocgcolorlinks,allcolors={blue}]{hyperref}
\usepackage{xcolor}
\usepackage{authblk}
\usepackage{bm}

\usepackage{stackengine}
\usepackage{newtxtext,newtxmath}
\usepackage{cite}
\newcommand{\email}[1]{\href{mailto:#1}{#1}}

\newtheorem{theorem}{Theorem}
\newtheorem{prop}[theorem]{Proposition}

\theoremstyle{remark}
\newtheorem{remark}[theorem]{Remark}
\theoremstyle{definition}
\newtheorem{assumption}{Assumption}
\newtheorem{definition}[theorem]{Definition}

\newcommand{\llbrace}{\lbrace\hspace{-0.15cm}\lbrace}
\newcommand{\rrbrace}{\rbrace\hspace{-0.15cm}\rbrace}
\newcommand{\norm}[1]{\| #1\|}
\newcommand{\trinorm}[1]{{\vert\kern-0.25ex\vert\kern-0.25ex\vert #1 \vert\kern-0.25ex\vert\kern-0.25ex\vert}}

\title{On mathematical and numerical modelling of multiphysics wave propagation with polytopal Discontinuous Galerkin methods}
\author[1]{Paola F. Antonietti \footnote{\email{paola.antonietti@polimi.it}}}
\author[1]{Michele Botti \footnote{\email{michele.botti@polimi.it}}}	
\author[1]{Ilario Mazzieri \footnote{\email{ilario.mazzieri@polimi.it}}}
\affil[1]{MOX-Laboratory for Modeling and Scientific Computing, Department of Mathematics, Politecnico di Milano, Italy}

\begin{document}
\maketitle

\begin{abstract}
In this work we review discontinuous Galerkin finite element methods on polytopal  grids (PolydG) for the numerical simulation of multiphysics wave propagation phenomena in heterogeneous media. 
In particular, we address wave phenomena in elastic, poro-elastic, and poro-elasto-acoustic materials.  
Wave propagation is modeled by using either the elastodynamics equation in the elastic domain, the  acoustics equations in the acoustic domain and the low-frequency Biot's equations in the poro-elastic one. 
The coupling between different models is realized by means of (physically consistent) transmission conditions, weakly imposed at the interface between the subdomains.
For all models configuration, we introduce and analyse the PolydG semi-discrete formulation, which is then coupled with suitable time marching schemes. For the semi-discrete problem, we present the stability analysis and derive a-priori error estimates in a suitable energy norm. A wide set of verification tests with manufactured solutions are presented in order to validate the error analysis. Examples of physical interest are also shown to demonstrate the capability of the proposed methods.
\end{abstract}


\section{Introduction}

Multiphysics wave propagation in heterogeneous media is a very attractive research topic and, in recent decades, it has registered considerable interest in the mathematical, geophysical and engineering communities. 
Mathematical models for wave propagation phenomena range from the linear transport equation, to the non-linear system of Navier-Stokes equations. They appear in many different scientific disciplines, such as acoustic engineering \cite{TKWF2010}, vibroacoustics \cite{krishnan}, aeronautical engineering,  \cite{castagnede1998ultrasonic}, biomedical engineering \cite{HAIRE1999291}, and computational geosciences; see \cite{carcione2014book} for a comprehensive review. 

Thanks to the ongoing development of increasingly advanced high-performance computing facilities, the use of digital twins models for solving wave propagation problems has given a notable impulse towards a deeper understanding of these phenomena. 
Numerical methods designed for wave simulations must simultaneously account for the following three distinguishing features: \textit{accuracy}, \textit{geometric flexibility} and \textit{scalability}.
\textit{Accuracy} is essential to correctly reproduce the physical phenomenon, and allows to minimize numerical dispersion and dissipation errors that would deteriorate the quality of the solution. \textit{Geometric flexibility} is required since the computational domain usually features complicated geometrical shapes as well as sharp media contrasts. \textit{Scalability} is demanded to solve on parallel machines real computational models featuring several hundred of millions or even billions of unknowns. 

In this work we consider wave propagation problems arising from geophysics and we discuss and analyze several models, with increasing complexity, employed in this scientific area. We first present models of elastodynamics, then of poro-elasticity, and finally coupled poro-elasto-acoustics models.  

Elastodynamics and viscoelastodynamics models are typically used for the study of seismic waves that propagate across the globe and are generated by earthquakes, volcanic activity, or artificial explosions. As far as the elastodynamics equations are concerned, the most used numerical methods are finite differences \cite{Chaljub2015,Pitarka2021,McCallen2021}, finite elements \cite{Bielak2005}, finite volumes \cite{dumbser2007arbitrary,pelties2012three,duru2020stable,breuer2017edge}, and spectral elements in either their conforming \cite{komatitsch2004simulations,stupazzini2009near,Galvez2014} or discontinuous setting \cite{DeBasabe2008,Antonietti2018,ferroni2016dispersion}.

Poro-elastodynamics models are used to describe the propagation of pressure and elastic waves through a porous medium. Pressure waves propagate through the saturating fluid inside pores, while elastic ones through the porous skeleton. 
In the pioneering work by Biot \cite{biot1} general equations of waves propagation in poro-elastic materials were introduced. More recently, in \cite{smeulders1992wave} it is proposed a model of seismic waves in saturated soils, distinguishing between in-phase (\textit{fast}) movements between solid and fluid and out-phase (\textit{slow}) ones. Poro-elasto-acoustic problem model acoustic/sound waves impacting a porous material and consequently propagating through it. The coupling between acoustic and poro-elastic domains, realized by means of physically consistent transmission conditions at the interface, is discussed in \cite{gurevich1999interface} and \cite{chiavassa_lombard_2013}.

There is a wide strand of literature concerning the numerical discretization of poroelastic or poro-elasto-acoustic models. Here, we recall, e.g., the Lagrange Multipliers  method \cite{rockafellar1993lagrange,zunino,Flemisch2006}, the Finite Element method \cite{BERMUDEZ200317,FKTW2010}  the Spectral and Pseudo-Spectral Element method \cite{Tromp_Morency_2008,Sidler2010}, the ADER scheme \cite{delapuente2008,chiavassa_lombard_2013}, the finite difference method \cite{LOMBARD200490}, and references therein.

The aim of this work is to introduce and analyze a discontinuous Galerkin method on polygonal/polyhedral (polytopal, for short) grids (PolydG) for the numerical discretization of multiphysics waves propagation through heterogeneous materials. The geometric flexibility and arbitrary-order accuracy featured by the proposed scheme are crucial within this context as they ensure the possibility of handling complex geometries and an intrinsic high-level of precision that are necessary to correctly represent the solutions.

For early results in the field of dG methods we refer the reader to \cite{AnBrMa2009,BaBoCoDiPiTe2012,AntoniettiGianiHouston_2013,CangianiGeorgoulisHouston_2014,CangianiDongGeorgoulisHouston_2016,CongreveHouston2019,cangiani2020hpversion}
for second-order elliptic problems, to \cite{CangianiDongGeorgoulis_2017} for 
parabolic differential equations, to \cite{AntoniettiFacciolaRussoVerani_2019} for flows in fractured porous media, and to \cite{AntoniettiVeraniVergaraZonca_2019} for fluid structure interaction problems.
In the framework of dG methods for hyperbolic problems, we mention \cite{riviere2003discontinuous,GrScSc06} for scalar wave equations on simplicial grids and the more recent PolydG discretizations designed in \cite{AntoniettiMazzieri_2018} for elastodynamics problems, in \cite{AntoniettiMazzieriMuhrNikolicWohlmuth_2020} for non-linear sound waves, in \cite{bonaldi,AntoniettiBonaldiMazzieri_2019b} for coupled elasto-acoustic problems, and in \cite{Antonietti.ea:21} for poro-elasto-acoustic wave propagation. 

The results of the present work review and extend the analysis conducted in \cite{AntoniettiMazzieri_2018} and \cite{Antonietti.ea:21}. In particular, in Section \ref{sec:poro-el} we provide a novel stability estimate for the poro-elastic case requiring minimal regularity on problem data and showing explicitly the dependence on the model coefficients, final simulation time, and initial conditions. Additionally, in Section \ref{sec:poro-el-ac} we generalize the PolydG semi-discrete formulation derived in \cite[Section 3]{Antonietti.ea:21} to the heterogeneous case, namely we allow all the model coefficients to be discontinuous over the computational domain.

The remaining part of the paper is structured as follows: in Section \ref{sec:model_problem} we review the differential models for wave propagation in heterogeneous Earth's media while in Section \ref{sec:DGnotation} we define the discrete setting used in the paper.  
The elastodynamic model and its numerical discretization through a PolydG method is recalled in Section \ref{sec:linear_elasticity}, while Sections \ref{sec:poro-el} and \ref{sec:poro-el-ac} discuss the numerical analysis of a PolydG method for wave propagation problems in poro-elastic and coupled poro-elastic-acoustic media, respectively. 
Section \ref{sec:linear_elasticity}, \ref{sec:poro-el}, and \ref{sec:poro-el-ac} contain at the end suitable verification test cases to validate the theoretical error bounds.
Different numerical tests of physical interest are introduced and discussed in Section \ref{sec:NumResults}. Finally, in Section \ref{sec:conclusions} we draw some conclusions and discuss some perspective about future work.        

\subsection*{Notation}

In the following, for an open, bounded domain $D\subset\mathbb{R}^d, d=2, 3$, the notation $\bm{L}^2(D)$ is used in place of $[L^2(D)]^d$, with $d\in\{2,3\}$. The scalar product in $L^2(D)$ is denoted by $(\cdot,\cdot)_D$, with associated norm $\norm{\cdot}_D$.  Similarly, $\bm{H}^\ell(D)$ is defined as $[H^\ell(D)]^d$, with $\ell\geq 0$, equipped with the norm $\norm{\cdot}_{\ell,D}$, assuming conventionally that $\bm{H}^0(D)\equiv\bm{L}^2(D)$. In addition, we will use $\bm{H}(\textrm{div},D)$ to denote the space of $\bm{L}^2(D)$ functions with  square integrable divergence. In order to take into account essential boundary conditions, we also introduce the subspaces 
\begin{align}
H^1_0(D) &= \{ \psi\in H^1(D) \, \mid \, \psi_{\mid \Gamma_{D}} = 0 \}, \\
\bm{H}^1_0(D) &= \{ \bm{v}\in \bm{H}^1(D) \,\mid \, \bm{v}_{\mid \Gamma_{D}} = \bm{0}\}, \\
\bm{H}_0(\textrm{div},D) &= \{ \bm{z}\in \bm{H}(\textrm{div},D) \,\mid \, (\bm{z}\cdot\bm{n}_p)_{\mid\Gamma_{D}} = 0\},
\end{align}
with $\Gamma_D\subset\partial D$ having strictly positive Hausdorff measure.
Given $k\in\mathbb{N}$ and a Hilbert space $\mathbb{H}$, the usual notation $C^k([0,T];\mathbb{H})$ is adopted for the space of $\mathbb{H}$-valued functions, $k$-times continuously differentiable in $[0,T]$.  Finally, the notation $x\lesssim y$ stands for $x\leq C y$, with $C>0$ independent of the discretization parameters, but possibly dependent on the physical coefficients and the final time $T$.


\section{Modelling seismic waves}\label{sec:model_problem}

A seismic event is the result of a sudden release of energy due to the rupture of a more fragile part of the Earth's crust called the fault. The deformation energy, accumulated for tens and sometimes hundreds of years along the fault, is transformed into kinetic energy that radiates, in the form of waves, in all directions through the layers of the Earth.
Seismic waves are therefore energy waves that produce an oscillatory movement of the ground during their passage. Seismic waves are subdivided into two main categories: volume waves and surface waves. The former can be decomposed into compression waves (P) and shear waves (S). The (faster) P waves are transmitted both in liquids and in solids, while the (slower) S waves travel only in solid media. P waves induce a ground motion aligned with the wave field direction while S waves induce ground a motion in a plane perpendicular to the wave propagation field.

More and more frequently mathematical models are used for the study and analysis of ground motion. The solution of these models through appropriate numerical methods can provide important information for the evaluation of the seismic hazard of a given region and for the planning of the territory in order to limit the socio-economic losses linked to the seismic event.
In the following we consider the differential model that aims at describing the propagation of seismic wave within Earth's interior. 

Let $\Omega$ be a bounded domain modeling the portion of the Earth where the passage of seismic waves occurs, and let $\partial \Omega$ be its boundary that can be decomposed into three disjoint parts $\Gamma_D,\, \Gamma_N$, and $\Gamma_{A}$. The values of the displacement (Dirichlet conditions), the values of tractions (Neumann conditions), and the values of fictitious tractions (absorbing conditions) are imposed on $\Gamma_D,\, \Gamma_N$, and $\Gamma_{A}$, respectively.
For a temporal interval $(0,T]$, with $T>0$, the equation governing the displacement field $\bm u( \bm x, t)$ of a dynamically disturbed elastic medium can be expressed as 
\begin{equation} \label{eq:WaveEquation}
\rho \partial_{tt}\bm{u} -  \nabla \cdot \bm \sigma    = \bm f \quad  \textrm{ in }\Omega \times (0,T],
\end{equation}
where $\rho$ is the mass density, $\bm f$ defines a suitable seismic source and $\bm \sigma$ is the stress tensor that models the constitutive behaviour of the material. Possible definition for $\bm \sigma$ and $\bm f$ will be discussed in the sequel. Equation \eqref{eq:WaveEquation} is completed by prescribing suitable boundary conditions as well as initial conditions. 
For the latter, by choosing ${\bm u}(\cdot,0) = \partial_t {\bm u}(\cdot,0) = \bm 0$, we suppose the domain to be at rest at the initial observation time.  

\subsection{Seismic waves in viscoelastic media}
The stress tensor $\bm \sigma$ in \eqref{eq:WaveEquation} can be defined in different ways to properly model the behavior of the soil. 
Before presenting the main constitutive laws that can adopted for seismic wave propagation analysis we introduce: (i) the strain tensor $\bm \epsilon$, defined as the symmetric gradient $\bm{\epsilon} (\bm u)= (\nabla \bm u + \nabla^T \bm u)/2$, and (ii) the fourth-order (symmetric and positive definite) stiffness tensor $\mathbb{D}$, encoding the mechanical properties of the medium. It is expressed in term of the first and the second  Lam\'e coefficients, namely $\lambda$ and $\mu$, respectively.  For an elastic material the generalized Hooke's law 
\begin{equation}\label{Hooke_law}
\bm \sigma = \mathbb{D} : \bm \epsilon    
\end{equation}
defines the most general linear relation among all the components of the stress and strain
tensor. In the most general case, i.e. a fully anisotropic material,  equation \eqref{Hooke_law} contains 21 material parameters. However in our case, i.e., for a perfectly isotropic material, \eqref{Hooke_law} can be reduced as
\begin{equation} \label{eq:stress}
\bm \sigma (\bm u) = \lambda \nabla \cdot \bm \epsilon (\bm u) \bm I  + 2 \mu \bm \epsilon(\bm u),
\end{equation}
where $\bm I$ is the identity tensor. 

Pure elastic constitutive laws are not physically representative in the field of application of interest.
A first model for visco-elastic media can be handled by modifying the equation of motion
according to \cite{KOSLOFF1986363}. In the approach, the inertial term  $\rho\partial_{tt} {\bm u}$
in \eqref{eq:WaveEquation} is replaced by $\rho\partial_{tt} {\bm u} + 2\rho\zeta\partial_t {\bm u} + \rho\zeta^2 {\bm u}$ where $\zeta$ is an attenuation parameter. As a matter of fact, with this substitution, i.e., 
\begin{equation}\label{viscoel_model}
\rho\partial_{tt} {\bm u} + 2\rho\zeta\partial_t {\bm u} + \rho\zeta^2 {\bm u} -  \nabla \cdot \bm \sigma    = \bm f \quad  \textrm{ in }\Omega \times (0,T],
\end{equation}
all frequency components are equally attenuated with distance, resulting in a frequency proportional quality factor $Q>0$ \cite{Morozov2008}. 
A second attenuation model is obtained by considering materials ``endowed with memory'' in the sense that
the state of stress at the instant $t$ depends on all the deformations undergone by the material in previous times. This behaviour can be expressed through an integral equation of the form 
\begin{equation}\label{eq:viscomodel2}
    \bm \sigma (t) = \int_0^t \frac{\partial \mathbb{D}} {\partial t} (t-s) : \bm{\epsilon} (s) \,ds,
\end{equation}
where the stress $\bm \sigma$ is  determined  by  the  entire  strain history. Implicit in this law is the dependence  on time of the Lam\'e parameters $\lambda$ and $\mu$, cf. \cite{komatitsch2002spectral,moczo2014finite}. We remark that, by using \eqref{eq:viscomodel2} it is possible to obtain an almost constant quality factor $Q$ in a suitable frequency range, cf. \cite{moczo2014finite}.
\\\\
We conclude this section by addressing proper boundary conditions to supplement equation \eqref{viscoel_model}. Several conditions can be set to correctly define the interaction between the wave and the domain boundary. 
Dirichlet conditions are employed to prescribe the behaviour of the displacement field, i.e. $\bm u = \bm{g}_D$ on $\Gamma_D$, while Neumann conditions $\bm \sigma  \bm n = \bm{g}_N$ on $\Gamma_N$ represent the distribution of surface loads. Here $\bm n$ denotes the outward pointing normal unit vector with respect to $\partial \Omega$. 

For geophysical applications, since the domain of interest $\Omega$ represents a portion of the Earth the following boundary conditions are commonly adopted: (i) free-surface condition, i.e.  $\bm \sigma \bm n = \bm 0$ for the top Earth's surface and (ii) transparent boundary conditions $\bm \sigma \bm n = \bm t$ for the remaining lateral and bottom surfaces. 
The latter consists in modeling the absorbing boundary layers by introducing a fictitious traction term $\bm t= \bm t(\bm u, \partial_t \bm u)$, consisting of a linear combination of displacement space and time derivatives. Examples can be found in \cite{Quarteroni1998,Givoli2010}. In Figure ~\ref{fig::dominio_elastico} we report an illustrative example of domain $\Omega$ together with boundary conditions. 

\begin{figure}
    \centering
    \includegraphics[width=0.5\textwidth]{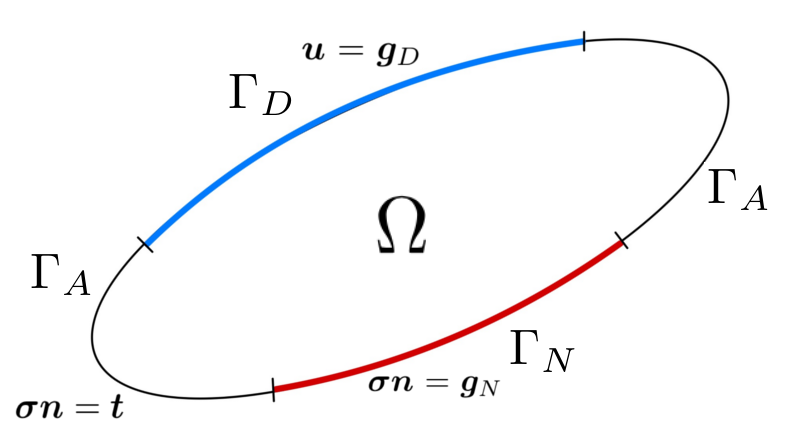}
    \caption{Example of domain $\Omega$ with boundary $\partial \Omega$ divided into a Dirichlet $\Gamma_D$, a Neumann  $\Gamma_N$ and an absorbing $\Gamma_A$ part.  }
    \label{fig::dominio_elastico}
\end{figure}

\subsection{Seismic waves in porous media}

Modeling wave propagation through fluid-saturated porous rock is crucial for the characterization of
the seismic response of geologic formations. In this case, the effects stemming from the interaction between the viscous fluid and the solid matrix have to be taken into account. In the framework of Biot's poro-elasticity theory \cite{biot1941general, biot1}, the total stress tensor $\widetilde{\bm \sigma}$ additionally depends on the pore pressure $p$ according to the following relation
\begin{equation}\label{eq:total_stress}
\widetilde{\bm \sigma}(\bm u, p) = 
\bm\sigma(\bm u) -\beta p \bm I,
\end{equation}
with $\bm\sigma(\bm u)$ defined as in \eqref{eq:stress} and $0<\beta\le 1$ denoting the Biot coefficient. 
Adding to the momentum balance equation \eqref{eq:WaveEquation}
the inertial term corresponding to the filtration displacement $\bm w = \phi({\bm w}_f - \bm u)$, where $\phi>0$ is the reference porosity and ${\bm w}_f$ the fluid displacement, leads to 
\begin{equation}\label{eq:balance_porel}
\rho\partial_{tt}\bm{u} + \rho_f\partial_{tt}\bm{w}
-\nabla\cdot\widetilde{\bm{\sigma}}=\bm{f} \qquad\text{in }\Omega\times(0,T].
\end{equation}
Here, the average density $\rho$ is given by $\rho=\phi\rho_f+(1-\phi)\rho_s$, where  $\rho_f>0$ is the saturating fluid density and $\rho_s>0$ is the solid density.
To derive Biot's wave equations in Section \ref{sec:poro-el}, the rheology of the porous material \eqref{eq:total_stress} and the momentum balance \eqref{eq:balance_porel} are combined with the dynamics of the fluid system described by Darcy's law and the conservation of fluid mass in the pores.

Two major differences have been observed when dealing with poro-elastic media instead of elastic ones: (i) the attenuation due to wave-induced fluid flow and (ii) the presence of an additional compressional wave of the second kind (slow P-wave), which becomes a diffusive mode in the low-frequency range, cf. \cite{carcione2014book}. As observed in \cite{delapuente2008}, this slow P-wave is mainly localized near the material heterogeneities or the source.

\subsection{Modelling the seismic source}

Seismic wave can be generated by different natural and artificial sources. Depending on the problem's configuration, one can consider a single point-source, an incidence plane wave or a finite-size rupturing fault. 

We can define a point-wise force $\bm{f}$ acting on a point $\bm{x}_0\in\Omega$ in the $i^{th}$ direction as
\begin{equation}\label{point-force}
	\bm{f}(\bm x,t) = f(t) \bm{e}_i \delta(\bm x-\bm{x}_0),
\end{equation}
where $\bm{e}_i $ is the unit vector of the $i^{th}$ Cartesian axis, $\delta(\cdot)$ is the delta distribution, and $f(\cdot)$ is a function of time. The expression of $f(\cdot)$ can be selected among different waveforms. Here, we report one of the most employed one, i.e. the Ricker wavelet \cite{ricker1953}, defined as 
\begin{equation} \label{eq:ricker}
f(t) =A_0 (1-2 \beta_p (t-t_0)^2) e^{ -\beta_p (t-t_0)^2}, \quad \beta_p = \pi^2 f_p^2,
\end{equation}
where $A_0$ is the wave amplitude, $f_p$ is the peak frequency of the signal and $t_0$ is a fixed reference time.

To define a vertically incident plane wave one can consider a uniform distribution of body forces along the plane $z=z_0$ of the form $\bm f(\bm x,t) = f(t) \bm{e}_i \delta(z-z_0)$. The latter generates a displacement in the $i^{th}$ direction given by 
\begin{equation} \label{eq:vert_pwave}
	\bar{u}_i(\bm x,t)=\frac{1}{2\rho c} H(t- \frac{\mid z-z_0\mid}{c})\int_0^{t - \frac{\mid z-z_0\mid}{c}}f(\tau) \,d\tau,
\end{equation}
where $H(\cdot)$ is the Heaviside function and $c$ (that can be equal either to  $c_P = \sqrt{\lambda+2\mu/\rho}$ or $c_S=\sqrt{\mu/\rho}$) is the wave velocity, see \cite{graff1975wave}. By taking the derivative of \eqref{eq:vert_pwave} with respect to time and evaluating the result at $z=z_0$ we can express $f(t)$ as $f(t) = 2 \rho c \frac{\partial \bar{u}_i}{\partial t}$.
Finally, we introduce one of the most important seismic input for seismic wave propagation that is the so called double-couple source force. A point double-couple or moment-tensor source localized in the computational domain $\Omega$ is often adopted to simulate small local or near-regional earthquakes.
Its mathematical representation is based on the seismic moment tensor $\bm m (\bm{x},t)$, defined in \cite{aki2002quantitative} as
\begin{equation*} 
m_{ij} (\bm x,t) = \frac{M_0 (\bm x,t)}{V} (s_{F,i} n_{F,j} + s_{F,j} n_{F,i}) \qquad i,j=1,..,d,
\end{equation*} 
where $\bm n_F$ and $\bm{s}_F$ denote the fault normal and the rake vector along the fault, respectively. $M_0 (\bm x,t)$ describes the time history  of the moment release at $\bm x$ and $V$ is the force elementary volume. The equivalent body force distribution is finally obtained through the relation $\bm f (\bm x,t)=-\nabla \cdot \bm{m}(\bm x,t)$, see \cite{faccioli19972d}. 


\section{Discrete setting for PolyDG methods}\label{sec:DGnotation}

In this section we define the notation related to the subdivision of the computational domain $\Omega$ by means of polytopic meshes.
We introduce a {polytopic} mesh $\mathcal{T}_h$ made of general polygons (in 2d) or polyhedra (in 3d). We denote such polytopic elements by $\kappa$, define by $\mid \kappa \mid$ their measure and by $h_\kappa$ their diameter, and set $h = \max_{\kappa \in \mathcal{T}_h} h_\kappa$.
We let a polynomial degree $p_\kappa\geq 1$ be associated with each element $\kappa\in\mathcal{T}_h$ and we denote by $p_h:\mathcal{T}_h\to\mathbb{N}^*=\{n\in\mathbb{N}\, \mid \, n\ge 1\}$ the piecewise constant function such that $(p_h)_{\mid\kappa}=p_\kappa$. The discrete space is introduced as follows: $\bm{V}_h=[\mathcal{P}_{p_h}(\mathcal{T}_h)]^d$, where $\mathcal{P}_{p_h}(\mathcal{T}_h)$ is the space of piecewise polynomials in $\Omega$ of total degree less than or equal to $p_\kappa$ in any $\kappa\in\mathcal{T}_h$.

In order to deal with polygonal and polyhedral elements, we define an \textit{interface} of $\mathcal{T}_h$ as the intersection of the $(d-1)$-dimensional faces of any two neighboring elements of $\mathcal{T}_h$. If $d=2$, an interface/face is a line segment and the set of all interfaces/faces is  denoted by $\mathcal{F}_h$.
When $d=3$, an interface can be a general polygon that we assume could be further decomposed into a set of planar triangles collected in the set  $\mathcal{F}_h$.  
We decompose the faces of $\mathcal{T}_h$ into the union of \textit{internal} ($i$) and \textit{boundary} ($b$) faces, respectively, i.e.:
$\mathcal{F}_h=\mathcal{F}_h^{i}\cup\mathcal{F}^{b}_h$. Moreover we further split the boundary faces as $\mathcal{F}^{b}_h = \mathcal{F}^{D}_h \cup \mathcal{F}^{N}_h \cup \mathcal{F}^{A}_h$, meaning that on $\mathcal{F}^{D}_h$ (resp. $\mathcal{F}^{N}_h$ and $\mathcal{F}^{A}_h$) Dirichlet (resp. Neumann and absorbing) boundary conditions are applied.

Following \cite{CangianiDongGeorgoulisHouston_2017}, we next introduce the main assumption on $\mathcal{T}_h$. 
\begin{definition}\label{def::polytopic_regular}
A mesh $\mathcal{T}_h$ is said to be \textit{polytopic-regular} if for any $ \kappa \in \mathcal{T}_h$, there exists a set of non-overlapping $d$-dimensional simplices contained in $\kappa$, denoted by $\{S_\kappa^F\}_{F\subset{\partial \kappa}}$, such that for any $F\subset\partial \kappa$, the following condition holds:
\begin{equation}
h_\kappa\lesssim d\mid S_\kappa^F \mid \, \mid F \mid^{-1}.
\label{eq::firstdef}
\end{equation}
\end{definition}
\begin{assumption}
The sequence of meshes $\{\mathcal{T}_h\}_h$ is assumed to be \textit{uniformly} polytopic regular in the sense of Definition \ref{def::polytopic_regular}.
\label{ass::regular}
\end{assumption}
\noindent
We remark that this assumption does not impose any restriction on either the number of faces per element nor their measure relative to the diameter of the element they belong to. Under Assumption \ref{ass::regular}, the following \textit{trace-inverse inequality} holds:
\begin{align}
& \norm{v}_{L^2(\partial \kappa)}\lesssim ph_\kappa^{-1/2}\norm{v}_{L^2(\kappa)}
&& \forall \ \kappa\in\mathcal{T}_h \ \forall v \in \mathcal{P}_p(\kappa),
\label{eq::traceinv}
\end{align}
see \cite[Section 3.2]{CangianiDongGeorgoulisHouston_2017} for the detailed proof and a complete discussion on inverse estimates.
In order to avoid technicalities, we also make the following \textit{$hp$-local bounded variation property} assumption.
\begin{assumption}
For any pair of neighboring elements $\kappa^\pm\in\mathcal{T}_h$, it holds $h_{\kappa^+}\lesssim h_{\kappa^-}\lesssim h_{\kappa^+}$ and $\ p_{\kappa^+}\lesssim p_{\kappa^-}\lesssim p_{\kappa^+}$.
\label{ass::3}
\end{assumption}

Next, following  \cite{Arnoldbrezzicockburnmarini2002}, for sufficiently piecewise smooth scalar-, vector- and tensor-valued fields $\psi$, $\bm{v}$ and $\bm{\tau}$, respectively, we define the averages and jumps on each \textit{interior} face $F\in\mathcal{F}_h^{i}$ shared by the elements $\kappa^{\pm}\in \mathcal{T}_h$ as follows:
\begin{align*}
 \llbracket\psi\rrbracket  &=  \psi^+\bm{n}^++\psi^-\bm{n}^-, 
 &&\llbracket\bm{v}\rrbracket  = \bm{v}^+\otimes\bm{n}^++\bm{v}^-\otimes\bm{n}^-, 
 &&\llbracket\bm{v}\rrbracket_{\bm n}  = \bm{v}^+\cdot\bm{n}^++\bm{v}^-\cdot\bm{n}^-, \\
 \llbrace\psi \rrbrace &= \frac{\psi^++\psi^-}{2}, 
&&\llbrace\bm{v}\rrbrace   = \frac{\bm{v}^++\bm{v}^-}{2}, 
&&\llbrace\bm{\tau}\rrbrace  = \frac{\bm{\tau}^++\bm{\tau}^-}{2}, 
\end{align*}
where $\otimes$ is the tensor product in $\mathbb{R}^3$, $\cdot^{\pm}$ denotes the trace on $F$ taken within the interior of $\kappa^\pm$, and $\bm{n}^\pm$ is the outward unit normal vector to $\partial \kappa^\pm$. Accordingly, on \textit{boundary} faces $F\in\mathcal{F}_h^{b}$, we set
$
\llbracket\psi\rrbracket = \psi\bm{n},\
\llbrace\psi \rrbrace = \psi,\
\llbracket\bm{v}\rrbracket= \bm{v}\otimes\bm{n},\
\llbracket\bm{v}\rrbracket_{\bm n} =\bm{v}\cdot\bm{n},\
\llbrace\bm{v}\rrbrace= \bm{v},\
\llbrace\bm{\tau}\rrbrace= \bm{\tau}.$ 
\\\\Finally, we introduced some important concepts employed for the convergence analysis of PolydG methods presented in the sequel, namely, the mesh covering $\mathcal{T}_{\sharp}$ and the Stein extension operator $\widetilde{\mathcal{E}}$. Indeed, the latter are used to extend standard $hp$-interpolation estimates on simplices to polytopal elements. We refer the reader to  \cite{CangianiGeorgoulisHouston_2014,Antonietti_et_al_2015,CangianiDongGeorgoulis_2017, CangianiDongGeorgoulisHouston_2017} for all the details. 

A covering $\mathcal{T}_{\sharp}=\{\mathcal{K}_\kappa\}$ related to the polytopic mesh $\mathcal{T}_h$ is a set of shape regular $d-$dimensional simplices $\mathcal{K}_\kappa$ such that for each $\kappa \in \mathcal{T}_h$ there exists a  $\mathcal{K}_\kappa \in \mathcal{T}_{\sharp}$ such that $\kappa \subset \mathcal{K}_\kappa$.  
We suppose that there exits a covering $\mathcal{T}_{\sharp}$
of $\mathcal{T}_{h}$ and a positive constant $C_\Omega$, independent of the mesh parameters, such that 
\begin{equation*}
    \max_{\kappa \in \mathcal{T}_h} card \{ \kappa' \in \mathcal{T}_h: \kappa' \cap \mathcal{K}_{\kappa} \neq \emptyset, \; \mathcal{K}_{\kappa} \in \mathcal{T}_{\sharp} \; \rm{s.t.} \; \kappa \subset \mathcal{K}_\kappa  \} \leq C_\Omega,
\end{equation*}
and $h_{T_\kappa} \lesssim h_\kappa$ for each pair $\kappa \in \mathcal{T}_\kappa$ and $T_\kappa \in \mathcal{T}_{\sharp}$ with $\kappa \subset \mathcal{T}_h$. This latter assumption assures that, when the computational mesh $\mathcal{T}_h$ is refined, the amount of overlap present in the covering $\mathcal{T}_{\sharp}$ remains bounded.

For an open bounded domain $\Sigma \subset\mathbb{R}^d$ and a polytopic mesh $\mathcal{T}_h$ over $\Sigma$ satisfying Assumption \ref{ass::regular}, we can introduce the Stein extension operator $\widetilde{\mathcal{E}}:H^m(\kappa)\rightarrow H^m(\mathbb{R}^d)$ \cite{stein1970singular}, for any $\kappa\in\mathcal{T}_h$ and $m\in\mathbb{N}_0$, such that $\tilde{\mathcal{E}}v\mid_\kappa=v$ and $\norm{\tilde{\mathcal{E}}v}_{m,\mathbb{R}^d}\lesssim\norm{v}_{m,\kappa}$. The corresponding vector-valued version mapping $\bm H^m(\kappa)$ onto $\bm H^m(\mathbb{R}^d)$ acts component-wise and is denoted in the same way. 
In what follows, for any $\kappa\in\mathcal{T}_h$, we will denote by $\mathcal{K}_\kappa$ the simplex belonging to $\mathcal{T}_{\sharp}$ such that $\kappa\subset\mathcal{K}_\kappa$.

\subsection{Time integration}\label{sec:time_int}
We introduce here the time integration scheme used for the numerical simulations shown in the following sections. First, we anticipate that, fixing a suitable basis for the discrete space, all the semi-discrete PolydG formulations that we will introduce in the following can be written in the general abstract form:
\begin{equation}\label{eq:2ndorder_abstract}
\text{For any time }\ t>0,\qquad\bm{M}_h\ddot{X_h}(t)+\bm{D}_h\dot{X_h}(t)+\bm{A}_h X_h(t)= S_h (t),
\end{equation} 
where the precise definition of the unknown $X_h$, the right-hand side $S_h$, and the matrices $\bm{M}_h, \bm{D}_h$, and $\bm{A}_h$ will be given in the forthcoming sections. Assuming that $\bm{M}_h$ is invertible, we have
\begin{equation}\label{eq:2ndorder_time}
\ddot{X_h}(t) =\bm{M}_h^{-1}(S_h(t) -\bm{D}_h\dot{X_h}(t) -\bm{A}_h X_h(t))
= \mathcal{L}_h(t,X_h,\dot{X}_h).
\end{equation}

Then, we discretize the interval $[0,T]$  by introducing a timestep $\Delta t>0$, such that $\forall \ k\in\mathbb{N}$, $t_{k+1}-t_k=\Delta t$ and define $X_h^k=X_h(t^k)$ and $Z^k_h=\dot{X}_h(t^k)$. Finally, to integrate in time \eqref{eq:2ndorder_abstract} we apply the Newmark$-\beta$ scheme defined by introducing a Taylor expansion for $X_h$ and $Z_h=\dot{X_h}$, respectively:
\begin{equation}
\begin{cases}
X^{k+1}_h=X^k_h +\Delta t Z^k_h +\Delta t^2(\beta_N\mathcal{L}^{k+1}_h + (\frac{1}{2}-\beta_N)\mathcal{L}^{k}_h),\\[5pt]
Z^{k+1}_h=Z^k_h +\Delta t(\gamma_N\mathcal{L}^{k+1}_h+(1-\gamma_N)\mathcal{L}^k_h),
\end{cases}
\label{eq::newmark_taylor}
\end{equation}
being $\mathcal{L}^k_h=\mathcal{L}_h(t^k, X^k_h, Z^k_h)$ and the Newmark parameters $\beta_N$ and $\gamma_N$  satisfy the constraints $0\leq\gamma_N\leq 1$, $0\leq2\beta_N\leq1$. The typical choices are $\gamma_N=1/2$ and $\beta_N=1/4$, for which the scheme is unconditionally stable and second order accurate. We also remark that, when $\mathcal{L}_h = \mathcal{L}_h(t^k, X^k_h)$, $\beta_N=0$, and $\gamma_N=1/2$, the Newmark scheme \eqref{eq::newmark_taylor} reduces to the leap-frog scheme which is explicit and second order accurate. We next address in detail the PolydG semi-discrete approximation of the problems we are considering.


\section{Elastic wave propagation in heterogeneous media}\label{sec:linear_elasticity}

Hereafter, for the sake of presentation, we will consider the linear visco elastodyamics model, i.e. equations \eqref{viscoel_model} and \eqref{eq:stress}. We suppose $\partial \Omega = \Gamma_D \cup \Gamma_N$ and we consider homogeneous Dirichlet and Neumann boundary conditions on $\Gamma_D$ and $\Gamma_N$, respectively.
The system of equations can be recast as 
\begin{equation} \label{eq:modelwave}
\begin{cases}
\rho\partial_{tt} {\bm u} + 2\rho\zeta\partial_t {\bm u} + \rho\zeta^2 {\bm u} - \nabla \cdot \bm \sigma = \bm{f} & \textrm{ in }\Omega \times (0,T], \\
\bm \sigma  = \mathbb{D} \bm{\epsilon}(\bm u) = \lambda \nabla \cdot \bm \epsilon (\bm u) \bm I  + 2 \mu \bm \epsilon(\bm u),
& \textrm{ in }\Omega \times (0,T], \\
\bm u = \mathbf{0} & \textrm{ on } \Gamma_D \times (0,T], \\
\bm{\sigma}  \, \bm n = \bm 0 & \textrm{ on }\Gamma_N  \times (0,T], \\
(\bm u, \partial_t\bm{u})  = (\bm{u}_0, \bm{u}_1) & \textrm{ in }\Omega \times \{0\}. 
\end{cases}
\end{equation}
The case with non homogenous Neumann conditions is treated in \cite{antonietti2016stability},
while absorbing conditions are considered in \cite{Quarteroni1998}. 
Finally, we refer to  \cite{Antonietti2018,riviere2007discontinuous} for a detailed analysis of viscoelastic attenuation models. 
We suppose the mass density $\rho$ and the Lam\'e parameters $\lambda$ and $\mu$ to be strictly positive  bounded functions of the space variable $\bm x$, i.e. $\rho,\lambda,\mu \in L^\infty(\Omega)$. We also suppose the forcing term $\bm f$ to be regular enough, i.e., $\bm{f} \in L^2 ((0, T]; \bm{L}^2 (\Omega))$ and that the initial conditions $(\bm{u}_0, \bm{u}_1) \in \bm{H}^1_0(\Omega) \times \bm{L}^2(\Omega)$. 
The weak formulation of problem~\eqref{eq:modelwave} reads as follows: for all $t \in (0,T]$ find $\bm u = \bm u(t)\in {\bm H}^{1}_{0}(\Omega)$ such that
\begin{equation}\label{el_weak:1p}
(\rho \partial_{tt} \bm{u}, \bm v )_\Omega + (2\rho\zeta\partial_t {\bm u}, \bm v)_{\Omega}  + (\rho\zeta^2 {\bm u}, \bm v)_\Omega +  \mathcal{A}^e(\bm u, \bm v) = (\bm f, \bm v)_\Omega \quad  \forall\, \bm v \in  {\bm H}^{1}_{0}(\Omega),
\end{equation}
where for any $\bm u,\bm v \in {\bm H}^{1}_{0}(\Omega)$ we have set 
\begin{equation}\label{eq:el_term}
\mathcal{A}^e(\bm u, \bm v) = 
( \bm{\sigma}(\bm u) , \bm{\epsilon}(\bm v))_\Omega.
\end{equation}
Problem~\eqref{el_weak:1p} is well-posed and  its unique solution $\bm u \in C((0,T]; {\bm H}^{1}_{0}(\Omega)) \cap C^1((0,T]; \bm{L}^{2}(\Omega) )$, see \cite[Theorem 8-3.1]{raviart1983introduction}.

\subsection{Semi-discrete formulation}\label{sec:dg_semidisc}

Using the notation introduced in Section~\ref{sec:DGnotation}, we define the PolyDG semi-discretization of  problem~\eqref{el_weak:1p}: for all $t \in (0,T]$, find $\bm{u}_h=\bm{u}_h(t)\in \bm{V}_h$ such that
\begin{equation}\label{el_elastodynamic_variational_formulation}
(\rho \partial_{tt} \bm{u}_{h}, \bm v_h)_\Omega +  (2\rho\zeta\partial_t {\bm u}_h, \bm v_h)_{\Omega}  + (\rho\zeta^2 {\bm u}_h, \bm v_h)_\Omega +  \mathcal{A}_h^e(\bm{u}_h, \bm{v}_h) = (\bm f, \bm v_h)_\Omega
\end{equation}
for any $\bm{v}_h\in \bm{V}_h$,
supplemented with the initial conditions $(\bm{u}_h(0),\partial_{t} {\bm u}_h(0))=(\bm{u}_h^0, \bm{u}_h^1)$,  where $\bm{u}_h^0, \bm{u}_h^1 \in \bm{V}_h$ are suitable approximations of $\bm{u}_0$ and $\bm{u}_1$, respectively. Here, we also assume the tensor $\mathbb{D}$ and the density $\rho$ to be element-wise constant over $\mathcal{T}_h$.
The bilinear form $\mathcal{A}_h^e: \bm{V}_h \times \bm{V}_h \rightarrow \mathbb{R}$ is defined as 	 
	\begin{multline}\label{el_B_tilde}
	\mathcal{A}_h^e(\bm u,\bm v) 
	= (\bm \sigma(\bm u) , \bm{\epsilon}(\bm v))_{\mathcal{T}_h}
	  -( \llbrace \bm \sigma(\bm u)\rrbrace,  \llbracket \bm v  \rrbracket)_{\mathcal{F}_h^i\cup\mathcal{F}_h^D}   \\
	  -( \llbracket \bm u  \rrbracket, \llbrace \bm \sigma(\bm v)\rrbrace)_{\mathcal{F}_h^i\cup\mathcal{F}_h^D}
	 +  (\eta\,\llbracket \bm u  \rrbracket, \llbracket \bm v  \rrbracket)_{\mathcal{F}_h^i\cup\mathcal{F}_h^D}	    
	\end{multline}
for all $\bm u,\bm v \in \bm V_h$. Here, we adopt the compact notation $(\cdot,\cdot)_{\mathcal{T}_h} = \sum_{T\in \mathcal{T}_h} (\cdot,\cdot)_T$ and $(\cdot,\cdot)_{\mathcal{F}_h^I \cup \mathcal{F}_h^D} = \sum_{F\in \mathcal{F}_h^I \cup \mathcal{F}_h^D} (\cdot,\cdot)_F$.
The penalization function $\eta:\mathcal{F}_h\rightarrow\mathbb{R}^+$  in \eqref{el_B_tilde} is defined face-wise as 
\begin{equation}\label{el_penalization_parameter}
\eta=\sigma_0  
\begin{cases}
\underset{\kappa\in\{\kappa_1,\kappa_2\} } \max \left(\overline{\mathbb{D}}_\kappa {p}^2_\kappa h_\kappa^{-1}\right), & F \in \mathcal{F}_h^i, \, F \subset \partial \kappa_1 \cap \partial \kappa_2, \\
\overline{\mathbb{D}}_\kappa {p}^2_\kappa h_\kappa^{-1}, &  F\in\mathcal{F}_h^D,  \, F \subset \partial \kappa \cap \Gamma_D.
\end{cases}
\end{equation}
where $\overline{\mathbb{D}}_\kappa=\mid (\mathbb{D}\mid_\kappa)^{1/2}\mid^2_2$ for any $\kappa \in \mathcal{T}_h$ (here $\mid\cdot\mid_2$ is the operator norm induced by the $l_2$-norm on ${\mathbb R}^n$, where $n$ denotes the dimension of the space of symmetric second-order tensors, i.e., $n=3$ if $d=2$, $n=6$ if $d=3$), and $\sigma_0$ is a (large enough) positive parameter at our disposal.

By fixing a basis for $\bm V_h$ and denoting by ${\rm U}_h$ the vector of the expansion coefficients in the chosen basis of the unknown $\bm{u}_h$, the semi-discrete formulation \eqref{el_elastodynamic_variational_formulation} can be written equivalently as:
\begin{equation}\label{eq:algebraicwave}
\bm M_\rho \ddot{\rm U}_h(t) + \bm M_{2\rho\zeta} \dot{\rm U}_h(t) +  (\bm A^e + \bm M_{\rho\zeta^2}) {\rm U}_h(t) = {\rm F}_h(t) \quad \forall t \in (0,T),
\end{equation}
with $\bm M$ denoting the mass matrix in $\bm{V}_h$, $\bm A^e$ the stiffness matrix corresponding to the bilinear form $\mathcal{A}^e$, and with initial conditions ${\rm U}_h(0) = \rm{U}_0$ and  $\dot{\rm U}_h(0) = \rm{U}_1$. Note that ${\rm F}_h$ is the vector representations of the linear functional $(\bm f , \bm v_h)_\Omega$.
Formulation \eqref{eq:algebraicwave} can be recast in the form \eqref{eq:2ndorder_abstract}-\eqref{eq:2ndorder_time} by setting $X_h={\rm U}_h$, $\bm{M}_h=\bm M_\rho$, $\bm{D}_h=\bm M_{2\rho\zeta}$, $\bm{A}_h= \bm A^e + \bm M_{\rho\zeta^2}$ and $S_h = {\rm F}_h$, cf. Section \ref{sec:time_int}.

\subsection{Stability and convergence results}

In this section we recall the stability and convergence results for the semidiscrete PolyDG formulation  \eqref{el_elastodynamic_variational_formulation}. We refer the reader to \cite{AntoniettiMazzieri_2018} and to \cite{Antonietti2021} for all the details.
The results are obtained  in the following energy norm
\begin{equation}\label{el_energy_norm} 
\| \bm u_h(t) \|^2_{\rm{E}} = \| \rho^{\frac12} {\partial_t{\bm u}}_h(t) \|^2_\Omega + \|\rho^{\frac12}\zeta {\bm u}_h(t) \|^2_\Omega +  \| \bm u_h(t) \|_{\textrm{DG},e}^2 \quad \forall  t\in (0,T],
\end{equation}
where 
\begin{equation}\label{el_dg_norm}
\| \bm v  \|_{\textrm{DG},e}^2  =  
\norm{\mathbb{D}^{\frac{1}{2}}\bm \epsilon (\bm v )}_{\mathcal{T}_h}^2
+  \norm{\eta^{\frac{1}{2}} \, \llbracket \bm v \rrbracket }_{\mathcal{F}_h^I \cup \mathcal{F}_h^D} ^2 \qquad
\forall \bm v  \in \bm V_h \oplus \bm H^1_0(\Omega),
\end{equation}
with $\|\cdot \|_{\mathcal{T}_h}^2 = (\cdot,\cdot)_{\mathcal{T}_h}$ and $\|\cdot \|_{\mathcal{F}_h^I \cup \mathcal{F}_h^D}^2 = (\cdot,\cdot)_{\mathcal{F}_h^I \cup \mathcal{F}_h^D}$.

\begin{prop}\label{el_prop:stability}
Let $\bm f\in L^2((0,T];\bm{L}^2(\Omega))$. Let $\bm u_h\in C^1((0,T];\bm V_h)$ be the approximate solution of~\eqref{el_elastodynamic_variational_formulation} obtained with the stability constant $\sigma_0$ defined in~\eqref{el_penalization_parameter} 
chosen sufficiently  large. Then,
\begin{equation}
\| \bm u_h(t) \|_{\textrm{E}} \lesssim \| \bm u_h^0 \|_{\textrm{E}} + \int\limits_{0}^{t}\| \bm f(\tau) \|_{\bm L^2(\Omega)}\, d\tau,    \qquad \forall0<t\leq T,
\label{eq:stability_elasticty}
\end{equation}
\end{prop}
\noindent where $\| \bm u_h(0) \|^2_{\rm{E}} = \| \rho^{\frac12} {\bm u}_{1,h} \|^2_\Omega + \|\rho^{\frac12}\zeta {\bm u}_{0,h} \|^2_\Omega +  \| \bm u_{0,h} \|_{\textrm{DG},e}^2$,
being ${\bm u}_{0,h},{\bm u}_{1,h} \in \bm V_h$ suitable approximation of the initial conditions $\bm u_0$ and $\bm u_1$, respectively.
\noindent
The proof of the previous stability estimate can be found for instance in \cite{AntoniettiMazzieri_2018,Antonietti2021}. 
From \eqref{eq:stability_elasticty} it is possible to conclude that the PolyDG approximation is dissipative. Indeed, when $\bm f = \bm 0$ (no external forces) the energy of the system at rest  $\| \bm u_h^0 \|_{\textrm{E}}$ is not conserved through time.  

Concerning the convergence results of the PolyDG scheme we report in the following the main result. We refer the reader to \cite{AntoniettiMazzieri_2018} for the details and for the proof of the following theorem. 

\begin{theorem} \label{thm:wave_eq_error}
Let Assumption \ref{ass::regular} and Assumption \ref{ass::3} be satisfied and assume that the exact solution $\bm u$ of~\eqref{el_weak:1p} is sufficiently regular. For any time $t \in [0,T]$, let $\bm u_h\in \bm V_h$ be the PolyDG solution of problem~\eqref{el_elastodynamic_variational_formulation} obtained with a penalty parameter $\sigma_0$ appearing in ~\eqref{el_penalization_parameter} sufficiently large. Then, for any time $t \in (0,T]$ the following bound holds
\begin{equation}\label{el_eq:error-estimate}
\begin{split}
\| \bm u-\bm u_h \|_{\textrm{E}}(t) 
&\lesssim \sum_{\kappa \in \mathcal{T}_h} \frac{h_\kappa^{s_\kappa-1}}{p_\kappa^{m_\kappa-3/2}} \left(
\mathcal{I}_{m_\kappa}^{\mathcal{T}_\sharp}(\bm u)(t) + \int_{0}^{t} \mathcal{I}_{m_\kappa}^{\mathcal{T}_\sharp}(\partial_t{\bm u})(s) \, ds\right),
\end{split}
\end{equation}
where
\begin{equation*}
    \mathcal{I}_{m_\kappa}^{\mathcal{T}_\sharp}(\bm u) = \| \widetilde{\mathcal{E}} \bm u \|_{\bm H^{m_\kappa}(\mathcal{T}_\sharp)} + \frac{h_\kappa}{p_\kappa^{3/2}}\| \widetilde{\mathcal{E}} (\partial_t{\bm u})\|_{\bm H^{m_\kappa}(\mathcal{T}_\sharp)} 
 + \| \widetilde{\mathcal{E}} \bm \sigma (\bm u)  \|_{ \bm{H}^{m_\kappa}(\mathcal{T}_\sharp)}
\end{equation*}
with $s_\kappa = \min(p_\kappa +1, m_\kappa)$ for all $\kappa \in \mathcal{T}_h$. The hidden constant depends on the material parameters and the shape-regularity of the covering $\mathcal{T}_\sharp$, but is independent of $h_\kappa$, $p_\kappa$.
\end{theorem}

\subsection{Verification test}\label{sec:ver_test_el}

We  solve the wave propagation problem \eqref{eq:modelwave} in $\Omega=(0,1)^2$, choosing $\lambda=\mu=\rho=\zeta=1$ and assuming that the exact solution $\bm u$ is given by
\begin{equation}
\bm u(\bm x,t) = \sin{(\sqrt{2}\pi t)}
\begin{bmatrix}
-\,\sin{(\pi x)^2} \sin{(2\pi y) }\\
\phantom{-}\sin{(2\pi x)} \sin{(\pi y)^2}
\end{bmatrix}.
\end{equation}
Dirichlet boundary conditions and initial conditions are set accordingly. We set the final time $T=1$ and chose a time step $\Delta t=10^{-4}$ of the leap-frog scheme, cf. \eqref{eq::newmark_taylor}. The penalty parameter $\sigma_0$ appearing in \eqref{el_penalization_parameter} has been set equal to 10.
We  compute the discretization error  by varying the polynomial degree $p_{\kappa}=p$, for any $\kappa \in \mathcal{T}_h$, and the number of polygonal elements $N_{el}$. 

\begin{figure}
    \centering
    \includegraphics[width=0.4\textwidth]{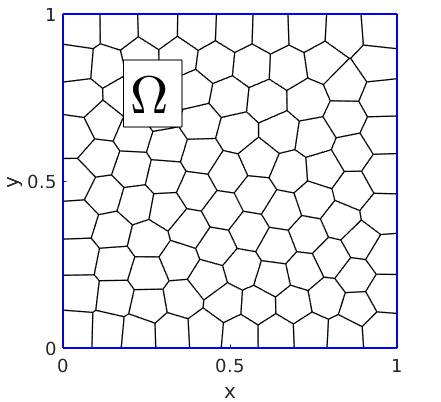}
    \caption{Test case of Section~\ref{sec:ver_test_el}. Example of computational domain having 100 polygonal elements.}
    \label{fig:domain_ver_el}
\end{figure}

In Figure~\ref{fig:convergence_elastic} (left), we report the computed energy error $\|\bm u-\bm u_h\|_{E}$ at final time $T$ as a function of the mesh size $h$. We retrieve the algebraic convergence proved in~\eqref{el_eq:error-estimate} for a polynomial degree $p=2,3,4$. 
Next, we report the computed  $\bm L^2$-error $\|e_{\bm u}\|_{\bm L^2(\Omega)} = \|\bm u-\bm u_h\|_{\bm L^2(\Omega)}$ at time $T$ obtained on a shape-regular polygonal grid (cf. Figure~\ref{fig:domain_ver_el}) versus the polynomial degree $p$, which varies from $1$ to $5$, in semilogarithmic scale. We fix the number of polygonal elements as $N_{el} = 160$. In this case we observe an exponential converge in $p$, as shown in Figure~\ref{fig:convergence_elastic} (right). 

\begin{figure}
\begin{minipage}{0.6\textwidth}
    \begin{tabular}{|c|c|c|c|}
    \hline 
    h \textbackslash \, p  &  2 & 3 & 4 \\
    \hline
     0.35   & 1.1078e+0 &  1.4101e-1 & 1.8809e-2 \\
     0.26   & 4.9112e-1 &  4.3925e-2 & 4.3186e-3 \\
     0.19   & 1.8714e-1 &  1.4661e-2 & 1.0703e-3 \\
     0.13   & 8.0198e-2 &  4.7140e-3 & 2.6145e-4  \\
     \hline 
     rate & 2.23 & 2.89 & 3.71 \\
     \hline
    \end{tabular} 
    \label{tab:my_label}
\end{minipage}
\hspace{3mm}
\begin{minipage}{0.4\textwidth}
\includegraphics[width=0.9\textwidth]{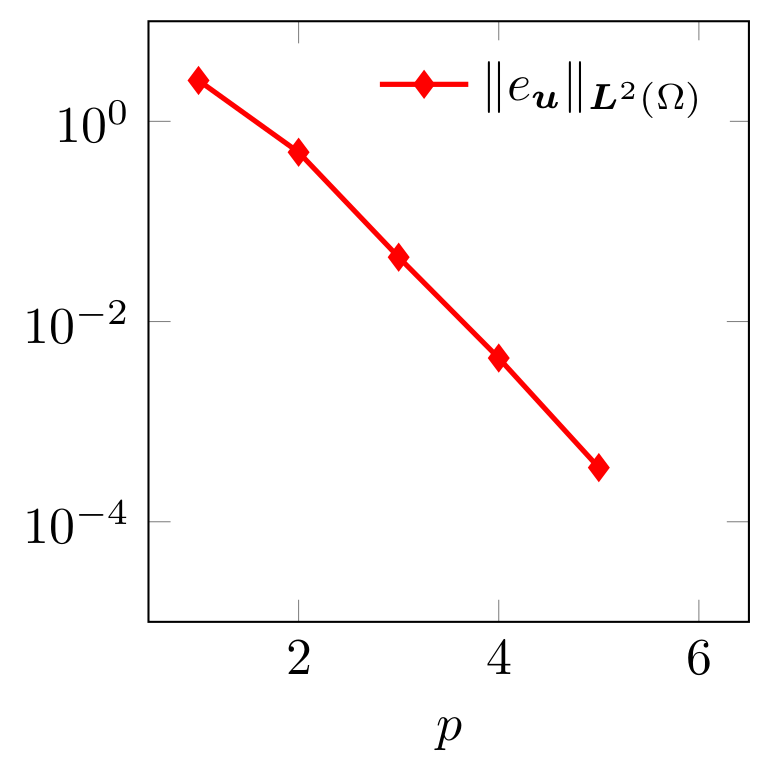}
\end{minipage}
    \caption{Test case of Section~\ref{sec:ver_test_el}. Computed energy error as a function of the mesh size $h$ for polynomial degree $p=2,3,4$. The computed rate of convergence is also reported in the last row, cf. \eqref{el_eq:error-estimate} (left). Computed $\bm L^2$-error as a function of the polynomial degree $p$ in a semilogarithmic scale  by fixing the  number of polygonal elements as $N_{el} = 160$ (right).}
    \label{fig:convergence_elastic}
\end{figure}


\section{Poro-elastic media} \label{sec:poro-el}

In this section, we consider a poro-elastic material occupying a polyhedral domain $\Omega_p\subset\Omega$ modeled by equations \eqref{eq:total_stress} and \eqref{eq:balance_porel}. The low-frequency Biot's system \cite{biot1} can be written as 
\begin{equation} \label{eq:model_porel}
\begin{cases}
\rho \partial_{tt} \bm{u} + \rho_f \partial_{tt}\bm{w} -\nabla\cdot\widetilde{\bm{\sigma}} = \bm{f} & \textrm{ in }\Omega_p \times (0,T], \\
\widetilde{\bm\sigma} = \lambda \nabla \cdot \bm \epsilon (\bm u) \bm I  + 2 \mu \bm \epsilon(\bm u) -\beta p \bm I
& \textrm{ in }\Omega_p \times (0,T], \\
\rho_f \partial_{tt} \bm{u} + \rho_w \partial_{tt}\bm{w} + \frac{\eta}{k}\dot{\bm{w}}
+\nabla p = \bm{g} & \textrm{ in }\Omega_p \times (0,T], \\
p = -m(\beta \nabla\cdot\bm{u}+\nabla\cdot \bm{w}) & \textrm{ in }\Omega_p \times (0,T], \\
\bm u = \mathbf{0}\qquad \text{and }\quad \bm w \cdot \bm n = 0 & \textrm{ on } \Gamma_{pD} \times (0,T], \\
\widetilde{\bm{\sigma}}  \, \bm n = \bm 0 \;\;\, \text{ and }\quad p=0 & \textrm{ on }\Gamma_{pN}  \times (0,T], \\
(\bm u, \partial_t\bm{u})  = (\bm{u}_0, \bm{u}_1) & \textrm{ in }\Omega_p \times \{0\}, \\ 
(\bm w, \partial_t\bm{w})  = (\bm{w}_0, \bm{w}_1) & \textrm{ in }\Omega_p \times \{0\},
\end{cases}
\end{equation}
where the density $\rho_w$ is given by $\rho_w=a \phi^{-1}\rho_f$ with tortuosity $a>1$, $\eta$ represents the dynamic viscosity of the fluid, $k$ is the absolute permeability, and $m$ denotes the Biot modulus. As in the previous section, we assume that the model coefficients $\rho_f, \rho_w, \eta k^{-1}, m \in L^\infty(\Omega_p)$ are strictly positive scalar fields and that the source term $\bm f,\bm g$ and the initial conditions $(\bm{w}_0, \bm{w}_1)$ are regular vector fields, namely $\bm f,\bm g\in L^2((0,T];\bm L^2(\Omega_p))$ and $(\bm{w}_0, \bm{w}_1)\in \bm H_0({\rm div},\Omega_p)\times\bm L^2(\Omega_p)$. The third and fourth equations in \eqref{eq:model_porel} correspond to the dynamic Darcy's law and the conservation of fluid mass, respectively. For the sake of simplicity, in \eqref{eq:model_porel} we have also assumed that the clamped region $\Gamma_{pD}\subset\partial\Omega_p$ is impermeable and a null pore pressure condition is prescribed on the Neumann boundary $\Gamma_{pN} = \partial\Omega_p\setminus\Gamma_{pD}$, cf. Figure \ref{fig::dominio_poroso}. We remark that more general boundary conditions can be treated up to minor modifications.

\begin{figure}
    \centering
    \includegraphics[width=0.5\textwidth]{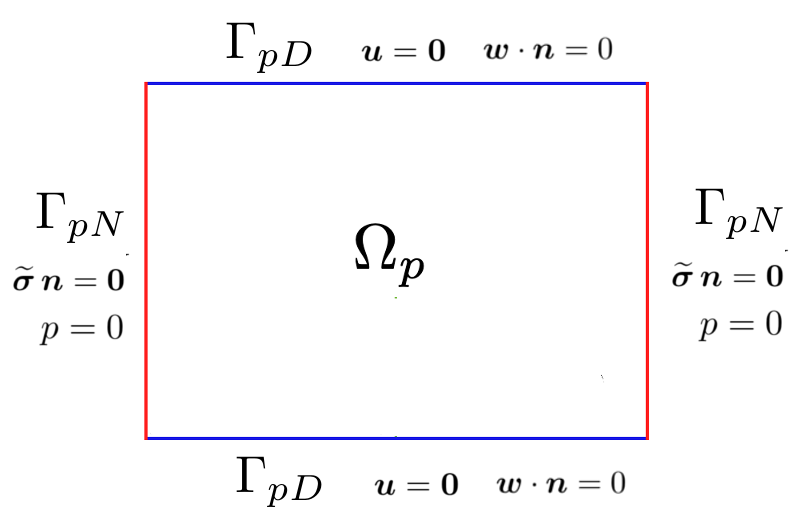}
    \caption{Example of a porous domain $\Omega_p$ together with mixed boundary conditions on $\Gamma_{pD}$ and $\Gamma_{pN}$.}
    \label{fig::dominio_poroso}
\end{figure}

In what follows, we focus on the two-displacement formulation  of the low frequency poro-elasticity problem \cite{matuszy}, that is obtained by inserting the expression of the total stress $\widetilde{\bm \sigma}$ and the pore pressure $p$ in the other equations in \eqref{eq:model_porel}. The corresponding weak formulation reads: for all $t \in (0,T]$ find $(\bm u (t), \bm w (t)) \in {\bm H}^{1}_{0}(\Omega_p)\times {\bm H}_{0}({\rm div},\Omega_p)$ such that
\begin{multline}\label{eq:porel_weak}
\mathcal{M}^p((\partial_{tt}\bm{u},\partial_{tt}\bm{w}), (\bm v, \bm z))
+(\eta k^{-1}\partial_{t}\bm w, \bm z)_{\Omega_p} +
\mathcal{A}^e(\bm u, \bm v)+ \mathcal{A}^p(\beta\bm{u}+\bm w, \beta\bm{v}+\bm z)\\
= (\bm f, \bm v)_{\Omega_p} + (\bm g, \bm z)_{\Omega_p}, \quad\forall (\bm v, \bm z) \in {\bm H}^{1}_{0}(\Omega_p)\times {\bm H}_{0}({\rm div},\Omega_p),
\end{multline}
with $\mathcal{A}^e:{\bm H}^{1}_{0}(\Omega_p)\times{\bm H}^{1}_{0}(\Omega_p)\to\mathbb{R}$ defined as the restriction to $\Omega_p$ of the function in \eqref{eq:el_term} and the bilinear forms $\mathcal{M}^p, \mathcal{A}^p$ defined as
\begin{equation}\label{eq:porel_term}
  \begin{aligned}
  \mathcal{M}^p((\bm{u},\bm{w}),(\bm v, \bm z)) &= 
  (\rho\bm{u}+\rho_f\bm{w},\bm v)_{\Omega_p} +(\rho_f\bm{u}+\rho_w\bm{w},\bm z)_{\Omega_p}, \\
  \mathcal{A}^p(\bm w, \bm z) &=
  (m\nabla\cdot\bm w, \nabla\cdot\bm{z})_{\Omega_p},
  \end{aligned}
\end{equation}
for all $(\bm u, \bm w), (\bm v, \bm z) \in {\bm H}^{1}_{0}(\Omega_p)\times {\bm H}_{0}({\rm div},\Omega_p)$.
The well-posedness of the low-frequency poro-elasticity problem \eqref{eq:porel_weak} has been established in \cite[Section 5.2]{ezziani} in the framework of semigroup theory.

\subsection{Semi-discrete formulation}\label{sec:semidisc_porel}

Proceeding as in Section \ref{sec:dg_semidisc}, we derive the semi-discrete PolydG approximation of problem \eqref{eq:porel_weak}. We introduce a polytopic mesh $\mathcal{T}_h^p$ of $\Omega_p$ satisfying Assumptions \ref{ass::regular} and \ref{ass::3} and denote by $\mathcal{F}_h^p$ the set of faces of $\mathcal{T}_h^p$.
Here, we consider the same polynomial space for both the discrete solid displacement $\bm u_h$ and filtration displacement $\bm w_h$, i.e. $\bm u_h,\bm w_h\in\bm V_h^p=(\mathcal{P}_{p_h}(\mathcal{T}_h^p))^d$, and we assume that all the model coefficients are piecewise constant over $\mathcal{T}_h^p$. The PolydG semi-discrete problem consists in finding, for all $t\in(0,T]$, the solution $(\bm u_h (t), \bm w_h (t)) \in \bm{V}_h^p\times\bm{V}_h^p$ such that 
\begin{multline}\label{eq:porel_discrete}
\mathcal{M}^p((\partial_{tt}\bm{u}_h,\partial_{tt}\bm{w}_h), (\bm v_h, \bm z_h))
+(\eta k^{-1}\partial_{t}\bm w_h, \bm z_h)_{\Omega_p} +\mathcal{A}_h^e(\bm u_h, \bm v_h) 
\\ + \mathcal{A}_h^p(\beta\bm{u}_h+\bm w_h,\beta\bm{v}_h+\bm{z}_h)
= (\bm f, \bm v_h)_{\Omega_p} + (\bm g, \bm z_h)_{\Omega_p}, \quad\forall \bm v_h, \bm z_h \in {\bm V}_h^p,
\end{multline}
where $\mathcal{A}_h^e:{\bm V}_h^p\times{\bm V}_h^p\to\mathbb{R}$ is defined as in \eqref{el_B_tilde} and the bilinear form $\mathcal{A}_h^p$ defined such that 
\begin{multline}\label{eq:Ahp}
\mathcal{A}_h^p(\bm{w},\bm{z})=
( m\nabla\cdot\bm{w},\nabla\cdot\bm{z})_{\mathcal{T}_h^p}
-(\llbrace m(\nabla\cdot\bm{w})\rrbrace, \llbracket\bm{z}\rrbracket_{\bm n})_{\mathcal{F}_h^{pi}\cup\mathcal{F}_h^{pD}} \\
- (\llbracket\bm{w}\rrbracket_{\bm n},\llbrace m(\nabla\cdot\bm{z})\rrbrace)_{\mathcal{F}_h^{pi}\cup\mathcal{F}_h^{pD}} +(\gamma \llbracket\bm{w}\rrbracket_{\bm n}, \llbracket\bm{z}\rrbracket_{\bm n})_{\mathcal{F}_h^{pi}\cup\mathcal{F}_h^{pD}},
\end{multline}
for all $\bm w, \bm z \in {\bm V}_h^p$ and the penalization function $\gamma\in L^\infty(\mathcal{F}_h^p)$ is given by
\begin{equation}\label{pen_m_parameter}
\gamma = m_0 
\begin{cases}
\underset{\kappa\in\{\kappa_1,\kappa_2\} } \max \left(m_{\mid\kappa} {p}^2_\kappa h_\kappa^{-1}\right), & F \in \mathcal{F}_h^{pi}, \, F \subset \partial \kappa_1 \cap \partial \kappa_2, \\
m_{\mid\kappa} {p}^2_\kappa h_\kappa^{-1}, &  F\in\mathcal{F}_h^{pD},  \, F \subset \partial \kappa \cap \Gamma_{pD},
\end{cases}
\end{equation}
where $m_0$ is a positive user-dependent parameter. We remark that, owing to the $\bm{H}({\rm div})$-regularity of the filtration displacement $\bm{w}$ solving \eqref{eq:porel_weak}, the penalization term in \eqref{eq:Ahp} acts only on the normal component of the jumps. 
Problem \eqref{eq:porel_discrete} is completed with suitable initial conditions $(\bm{u}_h(0),\bm{w}_h(0),\partial_{t} {\bm u}_h(0), \partial_{t} {\bm w}_h(0))=(\bm{u}_h^0, \bm{w}_h^0, \bm{u}_h^1, \bm{w}_h^1)\in {\bm V}_h^p\times {\bm V}_h^p\times {\bm V}_h^p\times {\bm V}_h^p$. 

We conclude this section by observing that the algebraic representation of the semi-discrete formulation \eqref{eq:porel_discrete} is given by
\begin{equation}\label{eq:algebraic_poro}
\hspace{-2mm}
\left[ \begin{matrix} 
\bm{M}^p_\rho   & \bm{M}^p_{\rho_f} \\
\bm{M}^p_{\rho_f} & \bm{M}^p_{\rho_w}
\end{matrix} \right]  
\left[ \begin{matrix} 
\ddot{U}_h \\
\ddot{W}_h 
\end{matrix} \right] + 
\left[ \begin{matrix} 
0 & 0 \\
0 & \bm{M}_{{\eta}k^{-1}}
\end{matrix} \right]  
\left[ \begin{matrix} 
\dot{U}_h \\
\dot{W}_h 
\end{matrix} \right] 
+\left[ \begin{matrix} 
\bm{A}^e +  \bm{A}^p_{\beta^2} & \bm{A}^p_{\beta}\\
\bm{A}^p_{\beta} & \bm{A}^p 
\end{matrix} \right] 
\left[ \begin{matrix} 
{U}_h \\
{W}_h 
\end{matrix} \right] = 
\left[ \begin{matrix} 
F_h \\
G_h 
\end{matrix} \right] \hspace{-1mm},
\end{equation}
with $[U_h,W_h,\dot{U}_h,\dot{W}_h](0)=[U_0,W_0,U_1,W_1]$ and $[F_h, G_h]^{\rm T}$ corresponding to the vector representation of the right-hand side of \eqref{eq:porel_discrete}. 
Recalling the notation introduced in Section \ref{sec:time_int} and setting $X_h = [U_h, W_h]^{\rm T}$, $S_h =[F_h, G_h]^{\rm T}$, and 
$$
\bm{M}_h = \left[ \begin{matrix} 
\bm{M}^p_\rho   & \bm{M}^p_{\rho_f} \\
\bm{M}^p_{\rho_f} & \bm{M}^p_{\rho_w}
\end{matrix} \right], \quad
\bm{D}_h = \left[ \begin{matrix} 
0 & 0 \\
0 & \bm{M}_{{\eta}k^{-1}}
\end{matrix} \right], \quad
\bm{A}_h = \left[ \begin{matrix} 
\bm{A}^e +  \bm{A}^p_{\beta^2} & \bm{A}^p_{\beta}\\
\bm{A}^p_{\beta} & \bm{A}^p 
\end{matrix} \right],
$$
equation \eqref{eq:algebraic_poro} can be rewritten in the form \eqref{eq:2ndorder_abstract}.

\subsection{Stability and convergence results}

The aim of this section is to establish an a priori estimate for the solution of problem \eqref{eq:porel_discrete}. First, we define for all $\bm u, \bm w\in C^1([0,T];\bm{V}_h^p)$ the energy function 
\begin{multline}\label{eq:porel_norm}
    \| (\bm u, \bm w)(t) \|^2_{\mathcal{E}} = 
    \| \rho_u^{\frac12} {\partial_t{\bm u}}(t) \|^2_{\Omega_p} +
    \| (\rho_f\phi)^{\frac12} \partial_t(\bm u +\phi^{-1}\bm w)(t)\|_{\Omega_p}^2 
    + \| \bm u(t) \|_{\textrm{DG},e}^2 \\
    + \mid (\beta\bm u + \bm w)(t)\mid_{\textrm{DG},p}^2 + 
    \|(\eta/k)^{\frac12}\bm w (0)\|_{\Omega_p}^2 +
    \int_0^t \|(\eta/k)^{\frac12}\partial_t \bm w (s)\|_{\Omega_p}^2\,{\rm d}s,
\end{multline}
with $\rho_u = \frac{\rho_s(1-\phi)}2$, the norm $\|\cdot\|_{\textrm{DG},e}:\bm V_h^p\to\mathbb{R}^+$ defined as in \eqref{el_dg_norm} and
\begin{equation}\label{porel_dg_norm}
\mid \bm z \mid _{\textrm{DG},p}^2  =  
\norm{m^{\frac12} \nabla\cdot \bm z}_{\mathcal{T}_h^p}^2
+\norm{\gamma^{\frac{1}{2}} \,\llbracket\bm z\rrbracket_{\bm n} }_{\mathcal{F}_h^{pI} \cup \mathcal{F}_h^{pD}}^2 \qquad\forall \bm z  \in \bm V_h^p \oplus \bm H_0({\rm div},\Omega_p).
\end{equation}
One can easily check that $\max_{0\le t\le T}\| (\cdot,\cdot)(t) \|^2_{\mathcal{E}}$ defines a norm on $C^1([0,T];\bm{V}_h^p\times\bm V_h^p)$, cf. \cite[Remark 3.2]{Antonietti.ea:21}. We are now ready to derive the stability estimate for the PolydG semi-discretization.
\begin{prop}\label{prop:stability_porel}
Let $\bm f,\bm g\in L^2((0,T];\bm{L}^2(\Omega_p))$ and let $\bm u_h,\bm w_h\in C^1((0,T];\bm V_h^p)$ be the solutions of~\eqref{eq:porel_discrete} obtained with sufficiently large penalization parameters $\sigma_0$ and $m_0$. Let additionally assume that $\rho_u^{-1}, k \eta^{-1}\in L^\infty(\Omega_p)$. Then, it holds
\begin{equation*}
\max_{t\in[0,T]} \left\| (\bm u_h, \bm w_h)(t) \right\|_{\mathcal{E}}
\le \int_0^T \left\|\left(k/\eta\right)^{\frac12}\bm g(s) \right\|_{\Omega_p}^2 {\rm d}s + 
T\int_0^T\left\|\rho_u^{-\frac12}\bm f(s)\right\|^2_{\Omega_p} {\rm d}s +\mathcal{E}_0,
\end{equation*}
with
\begin{multline}\label{eq:initialcond_depend}
    \mathcal{E}_0 = \mathcal{E}_0(\bm u_h^0,\bm w_h^0,\bm u_h^1,\bm w_h^1) = \mathcal{M}^p((\bm u_h^1, \bm w_h^1), (\bm u_h^1, \bm w_h^1))+\mathcal{A}^e_h(\bm u_h^0, \bm u_h^0) \\ + \mathcal{A}^p_h(\beta\bm u_h^0+\bm w_h^0, \beta\bm u_h^0+\bm w_h^0) + \|(\eta/k)^{\frac12}\bm w_h^0\|_{\Omega_p}^2.
\end{multline}
\end{prop}
\begin{proof}
First, we observe that the bilinear form $\mathcal{M}^p$ is positive definite. Indeed, owing to the definition of the density functions $\rho,\rho_u$, and $\rho_w$ and since $\widetilde{a}=a-1>0$, for all $(\bm v,\bm z)\neq (\bm 0,\bm 0)$ one has
\begin{equation}
\begin{aligned}\label{eq:positive_M}
\mathcal{M}^p((\bm v, \bm{z}), (\bm v, \bm z)) &= 
2\left\|\rho_u^{\frac12} \bm v\right\|^2_{\Omega_p}+
\left\|(\rho_f\phi)^{\frac12}\left(\bm v+\frac{\bm z}{\phi}\right)\right\|^2_{\Omega_p} +
\left\| \frac{(\rho_f\widetilde{a})^{\frac12}\bm z}{\phi^{\frac12}}  \right\|^2_{\Omega_p} \\
&> 2\left\|\rho_u^{\frac12} \bm v\right\|^2_{\Omega_p} +
\left\|(\rho_f\phi)^{\frac12} \left(\bm v+\phi^{-1}\bm z\right)\right\|^2_{\Omega_p} > 0.
\end{aligned}
\end{equation}
Furthermore, if the stability parameters $\sigma_0$ and $m_0$ are chosen sufficiently large, the bilinear forms $\mathcal{A}^e_h$ and $\mathcal{A}^p_h$ are coercive (see \cite[Lemma A.3]{Antonietti.ea:21}), i.e., for all $\bm v_h, \bm z_h\in\bm V_h^p$ it holds
\begin{equation}
\begin{aligned}\label{eq:coercive_Ah}
\mathcal{A}^e_h(\bm v_h, \bm v_h)&\ge\|\bm v_h\|_{\textrm{DG},e}^2, \\
\mathcal{A}^p_h(\beta\bm v_h + \bm z_h, \beta\bm v_h + \bm z_h) &\ge  
\mid\beta\bm v_h + \bm z_h\mid_{\textrm{DG},p}^2.
\end{aligned}
\end{equation}
Then, taking $(\bm v_h, \bm z_h) = (\partial_t\bm u_h, \partial_t\bm w_h)$ in \eqref{eq:porel_discrete} and integrating in time between $0$ and $t\le T$, it is inferred that 
\begin{multline*}
\hspace{-2mm}
[\mathcal{M}^p((\partial_t\bm u_h, \partial_t\bm w_h), (\partial_t\bm u_h, \partial_t\bm w_h)) 
+\mathcal{A}^e_h(\bm u_h, \bm u_h) + \mathcal{A}^p_h(\beta\bm u_h + \bm w_h, \beta\bm u _h + \bm w_h)](t) \\ \hspace{2mm}
+2\int_0^t \left\|\left(\frac\eta k\right)^{\frac12} \partial_t \bm w_h(s) \right\|_{\Omega_p}^2\, {\rm d}s=
2\int_0^t (\bm f, \partial_t \bm u_h)_{\Omega_p} (s) + (\bm g, \partial_t \bm w_h)_{\Omega_p} (s)\, {\rm d}s +\widetilde{\mathcal{E}}_0,
\end{multline*}
with $\widetilde{\mathcal{E}}_0=\mathcal{M}^p((\bm u_h^1, \bm w_h^1), (\bm u_h^1, \bm w_h^1))+\mathcal{A}^e_h(\bm u_h^0, \bm u_h^0) + \mathcal{A}^p_h(\beta\bm u_h^0+\bm w_h^0, \beta\bm u_h^0+\bm w_h^0)$.
Now, using \eqref{eq:positive_M} and \eqref{eq:coercive_Ah} to infer a lower bound for the left-hand side of the previous identity and summing $\|(\eta/k)^{\frac12}\bm w_h^0\|_{\Omega_p}^2$ to both sides of the resulting inequality, we obtain
\begin{multline}\label{eq:basic_stab}
\left\| (\bm u_h, \bm w_h)(t) \right\|_{\mathcal{E}} + 
\left\|\rho_u^{\frac12} \partial_t\bm u_h(t)\right\|^2_{\Omega_p} +
\int_0^t \left\|\left(\eta/k\right)^{\frac12} \partial_t \bm w_h(s) \right\|_{\Omega_p}^2\, {\rm d}s \\
\le 2\int_0^t (\bm f, \partial_t \bm u_h)_{\Omega_p} (s) + (\bm g, \partial_t \bm w_h)_{\Omega_p} (s)\, {\rm d}s +\mathcal{E}_0,
\end{multline}
where $\mathcal{E}_0 = \widetilde{\mathcal{E}}_0 + \|(\eta/k)^{\frac12}\bm w_h^0\|_{\Omega_p}^2$ corresponds to the quantity defined in \eqref{eq:initialcond_depend}. Therefore, to conclude it only remains to bound the right-hand side of \eqref{eq:basic_stab}. To do so, we apply the Cauchy--Schwarz and Young inequalities to infer
$$
2\int_0^t (\bm g, \partial_t \bm w_h)_{\Omega_p} (s)\, {\rm d}s \le  
\int_0^t \left\|\left(\eta/ k\right)^{\frac12} \partial_t \bm w_h(s) \right\|_{\Omega_p}^2 {\rm d}s + 
\int_0^t \left\|\left(k/\eta\right)^{\frac12} \bm g(s) \right\|_{\Omega_p}^2 {\rm d}s
$$
and
$$
\begin{aligned}
2\int_0^t (\bm f, \partial_t \bm u_h)_{\Omega_p} (s)\, {\rm d}s &\le 
\frac1t \int_0^t\left\|\rho_u^{\frac12} \partial_t\bm u_h(s)\right\|^2_{\Omega_p} {\rm d}s + t\int_0^t\left\|\rho_u^{-\frac12} \bm f(s)\right\|^2_{\Omega_p} {\rm d}s \\
&\le \max_{s\in[0,t]} \left\|\rho_u^{\frac12} \partial_t\bm u_h(s)\right\|^2_{\Omega_p} + t \int_0^t\left\|\rho_u^{-\frac12} \bm f(s)\right\|^2_{\Omega_p} {\rm d}s.
\end{aligned}
$$
inserting the previous bounds into \eqref{eq:basic_stab} and taking the maximum over $t\in[0,T]$, yields the assertion.
\end{proof}
\begin{remark}
We observe that, proceeding as in \cite[Lemma 7]{Boffi.Botti.Di-Pietro:16}, it is possible to obtain a stability estimate for problem~\eqref{eq:porel_discrete} requiring $\mu^{-1}\in L^\infty(\Omega_p)$ together with $\bm f\in H^1((0,T],\bm L^2(\Omega_p))$ instead of $\rho_u^{-1}\in L^\infty(\Omega_p)$. The key step is based on estimating the term $\int_0^t (\bm f, \partial_t \bm u_h)_{\Omega_p}$ by using partial integration and the discrete Korn's first inequality \cite[Lemma 1]{Botti.Di-Pietro.Guglielmana:19}.
\end{remark}

For the sake of conciseness, we decide not to present here the convergence analysis for the PolydG formulation of the poro-elastic problem \eqref{eq:porel_discrete}. However, an error estimate can be readily deduced from Theorem \ref{eq:poelac_error-estimate} below, in the case in which the exact solution on the acoustic part of the domain is null.

\subsection{Verification test}\label{ver_test_poro}

We consider problem \eqref{eq:model_porel} in $\Omega_p = (-1,0) \times (0,1)$ and choose as exact solution 
 \begin{align}\label{solution_poroelastic}
& \bm{u}(x,y;t)=
\begin{pmatrix}
x^2 \cos(\frac{\pi x}{2}) \sin(\pi x)
 \\
x^2 \cos(\frac{\pi x}{2}) \sin(\pi x)
\end{pmatrix}
\cos(\sqrt{2}\pi t),
&&
\bm{w}(x,y;t)=-\bm{u}(x,y;t).
\end{align}
As before, Dirichlet boundary conditions and initial conditions are set accordingly. 
The model problem is solved on a sequence of polygonal meshes as the one shown in Figure~\ref{poroelastic-mesh-param} (left), with physical parameters shown in  Figure~\ref{poroelastic-mesh-param} (right). 
\begin{figure}
\begin{minipage}{0.45\textwidth}
\centering
\includegraphics[scale=0.4]{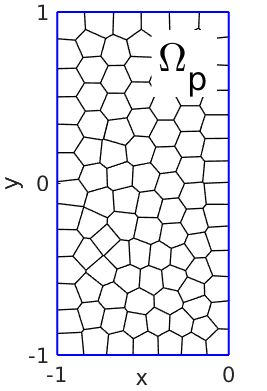}
\end{minipage}
\hfill
\begin{minipage}{0.5\textwidth}
\centering
\begin{tabular}{|c|c|}
\hline
\textbf{Field}      & \textbf{Value} \\ \hline
$\rho_f$, $\rho$ & 1                           \\ \hline
$\lambda$, $\mu$   & 1                           \\ \hline
$a$   & 1                           \\ \hline
$\phi$   & 0.5                           \\ \hline
$\eta$             & 1                       \\ \hline
$\rho_w$           & 2                       \\ \hline
$\beta$, m         & 1                       \\ \hline
\end{tabular}
\label{param}
\end{minipage}
    \caption{Poro-elastic test case of Section \ref{ver_test_poro}.  Polygonal mesh, with $N_{el} = 100$ polygons (left). Physical parameters (right).}\label{poroelastic-mesh-param}
\end{figure}
The final time $T$ has been set equal to $0.25$, considering a timestep of $\Delta t=10^{-4}$ for the Newmark-$\beta$ scheme,  $\gamma_N=1/2$ and $\beta_N=1/4$, cf. \eqref{eq::newmark_taylor}.
The penalty parameters $\sigma_0$ and $m_0$ appearing in definitions \eqref{el_penalization_parameter} and \eqref{pen_m_parameter}, respectively, have been chosen equal to 10. 
\begin{figure}
\begin{minipage}{0.6\textwidth}
    \begin{tabular}{|c|c|c|c|}
    \hline 
    h \textbackslash \, p  &  2 & 3 & 4 \\
    \hline
     0.36   & 5.8052e-1 &  1.0464e-1 & 1.1450e-2 \\
     0.25   & 3.3505e-1 &  3.1326e-2 & 2.9694e-3 \\
     0.18   & 1.7345e-1 &  1.1617e-2 & 8.0532e-4 \\
     0.13   & 8.9824e-2 &  4.7403e-3 & 2.0572e-4  \\
     \hline 
     rate & 2.10 & 3.06 & 3.86 \\
     \hline
    \end{tabular} 
\end{minipage}
\hspace{2mm}
\begin{minipage}{0.4\textwidth}
\includegraphics[width=0.9\textwidth]{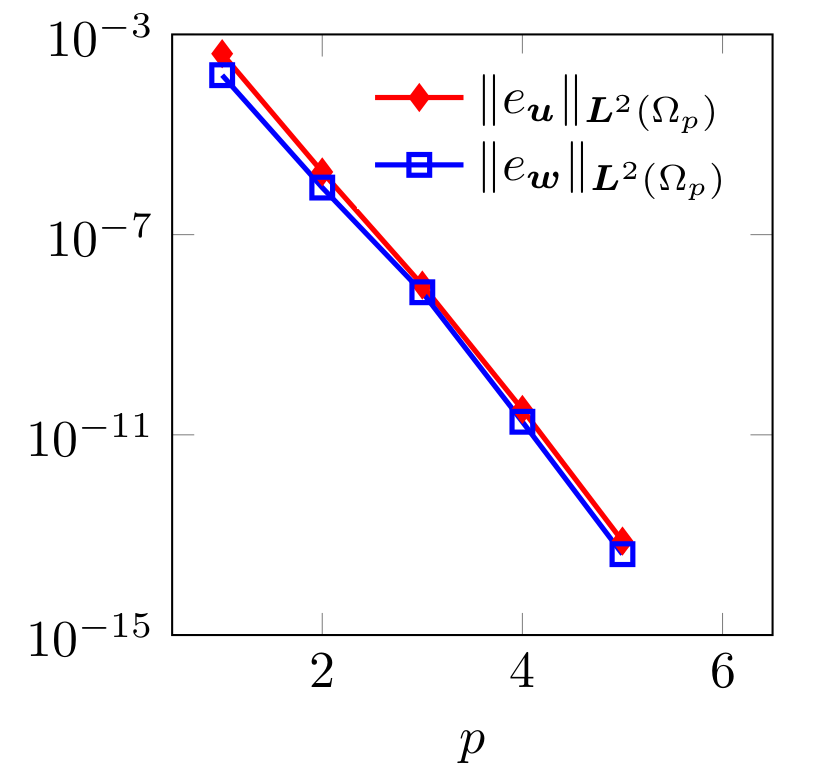}
\end{minipage}
    \caption{Test case of Section \ref{ver_test_poro}. Computed energy error as a function of the mesh size $h$ for polynomial degree $p=2,3,4$. The rate of convergence is also reported in the last row, cf. \eqref{eq:poelac_error-estimate} (left). Computed $\bm L^2$-errors $\|e_{\bm u}\|_{\bm L^2(\Omega_p)} = \|\bm u-\bm u_h\|_{\bm L^2(\Omega_p)}$ and $\|e_{\bm w}\|_{\bm L^2(\Omega_p)} = \|\bm w-\bm w_h\|_{\bm L^2(\Omega_p)}$  as a function of the polynomial degree $p$ in a semilogarithmic scale for $N_{el} = 100$  polygonal elements (right). }
    \label{fig:convergence_poro}
\end{figure}

In Figure~\ref{fig:convergence_poro} (left) we report the computed energy error $\|(\bm u-\bm u_h, \bm w-\bm w_h) \|_{\mathbb{E}}$, cf. \eqref{eq:poelac_error-estimate}, as a function of the mesh size $h$ for a polynomial degree $p=2,3,4$. In this case we retrieve the rate of convergence $\mathcal{O}(h^p)$ as proved in~\eqref{eq:poelac_error-estimate}. 
In Figure~\ref{fig:convergence_poro} (right) we plot the computed $\bm L^2$-errors for the elastic $\bm u$ and filtration $\bm w$ displacements as a function of the polynomial degree $p$ in a semilog-scale. We fix the number of polygonal elements as $N_{el}=100$. We observe an exponential rate of convergence since the solution \eqref{solution_poroelastic} is analytic.


\section{Poro-elastic-acoustic media}\label{sec:poro-el-ac}

In this section, we present the PolydG discretization of the poro-elasto-acoustic interface problem. We refer the reader to \cite{Antonietti.ea:21} for the rigorous mathematical analysis of the model problem and the detailed derivation of the proposed method.
In what follows, we assume that $\Omega$ is decomposed into two disjoint, polygonal/polyhedral subdomains: $\Omega=\Omega_p\cup\Omega_a$, cf. Figure \ref{fig::domain_pores}. 
\begin{figure}
\centering 
\includegraphics[width=0.45\textwidth]{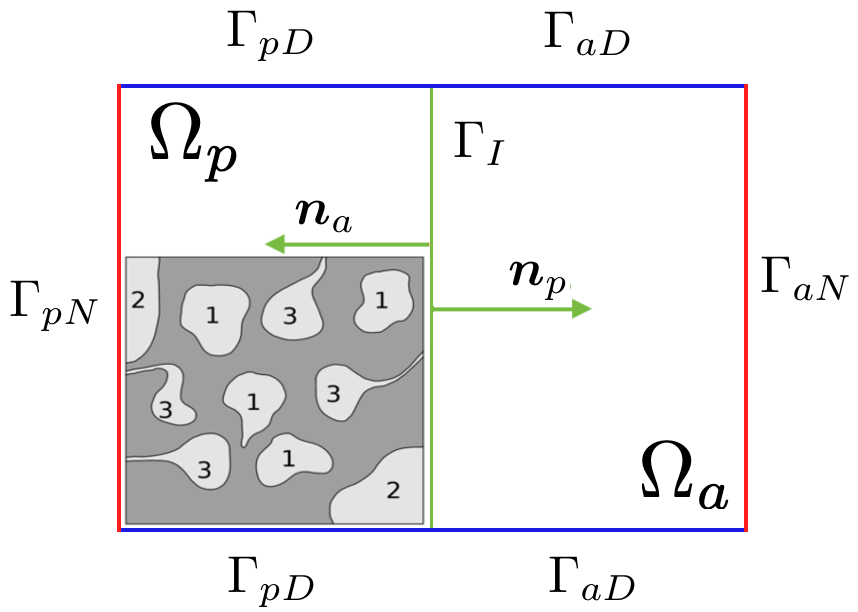}
\caption{\label{fig::domain_pores} Simplified representation of the domain $\Omega=\Omega_p\cup\Omega_a$ for $d=2$. Pores classification in $\Omega_p$:  \textit{sealed} (1), \textit{open} (2) and \textit{imperfect} (3).}
\end{figure}

The two subdomains share part of their boundary, resulting in the interface $\Gamma_I=\partial\Omega_p\cap\partial\Omega_a$. We set $\partial\Omega_p=\Gamma_{pD}\cup\Gamma_{pN}\cup\Gamma_I$ and $\partial\Omega_a=\Gamma_{aD}\cup\Gamma_{aN}\cup\Gamma_I$, where the surface measures of $\Gamma_{pD}$, $\Gamma_{aD}$, and $\Gamma_I$ are assumed to be strictly positive. The outer unit normal vectors to $\partial\Omega_p$ and $\partial\Omega_a$ are denoted by $\bm{n}_p$ and $\bm{n}_a$, respectively, so that $\bm{n}_p=-\bm{n}_a$ on $\Gamma_I$.

The subdomain $\Omega_p$ represents a poro-elasto medium whose dynamical behavior is described by Biot's equations \eqref{eq:model_porel}. In the fluid domain $\Omega_a$, we consider an acoustic wave with constant velocity $c>0$ and mass density $\rho_a>0$ such that $\rho_a, c^{-2}\in L^\infty(\Omega_a)$. 
For a given source term $h\in L^2((0,T]; L^2(\Omega_a))$, the acoustic potential $\varphi$ satisfies
\begin{equation} \label{eq::acousticeq}
\begin{cases}
\rho_a c^{-2}\partial_{tt}\varphi-\nabla \cdot(\rho_a \nabla\varphi) = h
& \textrm{ in }\Omega_a \times (0,T], \\
\varphi = 0 & \textrm{ on } \Gamma_{aD} \times (0,T], \\
\rho_a \nabla\varphi \cdot \bm n_a = 0 & \textrm{ on }\Gamma_{aN}  \times (0,T], \\
(\varphi, \partial_t\varphi)  = (\varphi_0, \varphi_1) & \textrm{ in }\Omega_a \times \{0\}, 
\end{cases}
\end{equation}
with $(\varphi_0, \varphi_1)\in H^1_0(\Omega_a)\times L^2(\Omega_a)$. 
To close the coupled poro-elasto-acoustic problem, some interface conditions on $\Gamma_I$ are needed. Here, we consider physically consistent transmission conditions (see, e.g., \cite{gurevich1999interface} and \cite{chiavassa_lombard_2011})
expressing the continuity of normal stresses, continuity of pressure, and conservation of mass:
\begin{equation}\begin{cases} \label{eq:interface}
-\widetilde{\bm{\sigma}}\bm{n}_p  = \rho_a\dot{\varphi}\bm{n}_p & \textrm{ on }\Gamma_{I}  \times (0,T],    \\
(\tau-1)\dot{\bm w}\cdot\bm{n}_p + \tau p   = \tau \rho_a\dot{\varphi} & \textrm{ on }\Gamma_{I}  \times (0,T],  \\
-(\dot{\bm{u}}+\dot{\bm{w}})\cdot\bm{n}_p = \nabla\varphi\cdot\bm{n}_p& \textrm{ on }\Gamma_{I}  \times (0,T]. 
\end{cases}\end{equation}
The parameter $\tau:\Gamma_I\to[0,1]$ denotes the hydraulic permeability at the interface and models different pores configurations, cf. Figure \ref{fig::domain_pores}. 
In the \textit{open pores} region $\tau^{-1}(1)\subset\Gamma_I$ the second equation in \eqref{eq:interface} reduces to $p=\rho_a\dot{\varphi}$, while in the \textit{sealed pores} subset  $\tau^{-1}(0)$ we have $\dot{\bm{w}}\cdot\bm{n}_p=0$, implying that $\tau^{-1}(0)$ is impermeable. Finally, the \textit{imperfect pores} region $\tau^{-1}((0,1))$ models an intermediate state between \textit{open} and \textit{sealed pores}. For later use, we split the interface into two disjoint (possibly non-connected) subsets $\Gamma_I = \Gamma_I^s\cup\Gamma_I^o$, with
$$
\Gamma_I^s = \tau^{-1}(0)\quad \text{and }\;\, 
\Gamma_I^o = \tau^{-1}((0,1]) = \Gamma_I\setminus\Gamma_I^s.
$$
We remark that the first and second conditions in \eqref{eq:interface} plays the role of a Neumann and a Robin-like conditions for system \eqref{eq:model_porel}, respectively. Similarly, the third equation in \eqref{eq:interface} acts as a Neumann condition for problem \eqref{eq::acousticeq}. The existence and uniqueness of a strong solution to the poro-elasto-acoustic problem coupling equations \eqref{eq:model_porel}, \eqref{eq::acousticeq}, and \eqref{eq:interface} is proved in \cite[Appendix A]{Antonietti.ea:21}.

In order to derive the weak formulation of the coupled problem, we introduce the function $\zeta_\tau:\Gamma_I^o\to\mathbb{R}^+$, defined by $\zeta_\tau= \tau^{-1}(1-\tau)$, and the weighted space
\begin{equation}\label{eq:W_tau}
\bm{W}_{\tau} = 
\{ \bm{z}\in \bm H_0(\textrm{div},\Omega_p) \,\mid\, \zeta_{\tau}^{\frac12}(\bm{z}\cdot\bm{n}_p)_{\mid\Gamma_I^o} \in L^2(\Gamma_I^o),\, (\bm{z}\cdot\bm{n}_p)_{\mid \Gamma_I^s} = 0 \},
\end{equation}
equipped with the norm 
$$
\norm{\bm{z}}_{\bm{W}_{\tau}} = \norm{\bm{z}}_{\Omega_p} + \norm{\nabla\cdot\bm{z}}_{\Omega_p} + 
\norm{\zeta_{\tau}^{\frac12}\ \bm{z}\cdot\bm{n}_p}_{\Gamma_I^o}
\qquad\forall\bm{z}\in\bm{W}_{\tau}.
$$
The weak form of the problem obtained by coupling equations \eqref{eq:model_porel}, \eqref{eq::acousticeq}, and \eqref{eq:interface} reads as: 
for any $t\in(0,T]$, find $(\bm{u},\bm{w},\varphi)(t) \in \bm H^1_0(\Omega_p) \times \bm{W}_{\tau} \times H^1_0(\Omega_a)$ s.t.
\begin{multline}\label{eq:coupled_weakform}
\mathcal{M}((\partial_{tt}\bm{u},\partial_{tt}\bm{w},\partial_{tt}\varphi), (\bm{v},\bm{z},\psi))  +
\mathcal{A}((\bm{u},\bm{w},\varphi),(\bm{v},\bm{z},\psi)) + \mathcal{B} (\partial_t \bm{w},\bm{z}) \\
+ \mathcal{C}(\partial_t\varphi,\bm{v} + \bm{z}) - \mathcal{C}(\partial_t(\bm{u} + \bm{w}), \psi)
= (\bm{f},\bm{v})_{\Omega_p}+(\bm{g},\bm{z})_{\Omega_p} +(h,\psi)_{\Omega_a}
\end{multline}
for all $(\bm{v},\bm{z},\psi) \in \bm H^1_0(\Omega_p) \times \bm{W}_{\tau} \times H^1_0(\Omega_a)$, where we have set
\begin{equation}\label{eq:bilinear_forms}
    \begin{aligned}
    \mathcal{M}((\bm{u},\bm{w},\varphi),(\bm{v},\bm{z},\psi))  & = \mathcal{M}^p((\bm{u},\bm{w}),(\bm{v},\bm{z})) 
    +(\rho_a c^{-2} \varphi, \psi)_{\Omega_a}, \\
    \mathcal{A}((\bm{u},\bm{w},\varphi),(\bm{v},\bm{z},\psi))  & = \mathcal{A}^e(\bm u, \bm v)
    +\mathcal{A}^p(\beta\bm{u}+\bm{w},\beta\bm{v}+\bm{z}) 
    +\mathcal{A}^a(\varphi,\psi), \\
	\mathcal{B} (\bm{w},\bm{z}) & = 
	(\eta k^{-1} \bm{w},\bm{z})_{\Omega_p}+ (\zeta_\tau\, \bm{w}\cdot\bm{n}_p, \bm{\bm{z}}\cdot\bm{n}_p)_{\Gamma_I^o}, \\
	\mathcal{C}(\varphi,\bm{z}) & =\langle \rho_a \varphi,\bm{z}\cdot\bm{n}_p\rangle_{\Gamma_I},
\end{aligned}
\end{equation}
with $\mathcal{M}^p,\mathcal{A}^e,\mathcal{A}^p$ defined as in \eqref{eq:porel_term}, \eqref{eq:el_term}, and \eqref{eq:Ahp}, respectively. In \eqref{eq:bilinear_forms}, the bilinear form $\mathcal{A}^a$ is defined such that $\mathcal{A}^a(\varphi,\psi)=(\rho_a\nabla \varphi,\nabla \psi)_{\Omega_a}$ for all $\varphi,\psi\in H^1_0(\Omega_a)$ and $\langle\cdot,\cdot\rangle_{\Gamma_I}$ denotes the $H^{\frac12}(\Gamma_I)$-$H^{-\frac12}(\Gamma_I)$ duality product.

\subsection{Semi-discrete formulation}

We decompose the polytopic regular mesh $\mathcal{T}_h$ as $\mathcal{T}_h=\mathcal{T}^p_h\cup\mathcal{T}^a_h$, where $\mathcal{T}_h^a$ and $\mathcal{T}_h^p$ are aligned with $\Omega_a$ and $\Omega_p$, respectively. In a similar way, we decompose $\mathcal{F}_h$ as $\mathcal{F}_h=\mathcal{F}_{h}^I \cup \mathcal{F}_h^p \cup \mathcal{F}_h^a$, where
$ \mathcal{F}_{h}^I=\{F\in\mathcal{F}_h:F\subset\partial \kappa^p\cap\partial \kappa^a,\kappa^p\in\mathcal{T}_{h}^p,\kappa^a\in\mathcal{T}_{h}^a\}$, and 
$\mathcal{F}_h^p$ and $\mathcal{F}_h^a$ denote the faces of $\mathcal{T}_h^p$ and $\mathcal{T}_h^a$, respectively, not laying on $\Gamma_I$.
The discrete spaces are selected as follows: given element-wise constant polynomial degrees $p_h:\mathcal{T}_h^p\to\mathbb{N}^*$ and $r_h:\mathcal{T}_h^a\to\mathbb{N}^*$, we let $\bm{V}_h^p=[\mathcal{P}_{p_h}(\mathcal{T}_h^p)]^d$ and $V_h^a=\mathcal{P}_{r_h}(\mathcal{T}_h^a)$. 
Finally, we also assume that the coefficients $\rho_a$ and $c$ are piecewise constant over $\mathcal{T}_h^a$ and $\tau$ is piecewise constant over $\mathcal{F}_h^I$. Under this assumption, we can decompose the set of mesh faces belonging to $\Gamma_I$ as $\mathcal{F}_h^I = \mathcal{F}_h^{Is}\cup\mathcal{F}_h^{Io}$, with $\mathcal{F}_h^{Is}=\{F\in\mathcal{F}_h^I\,\mid\, F\subset\Gamma_I^s\}$ and $\mathcal{F}_h^{Io}=\mathcal{F}_h^{I}\setminus\mathcal{F}_h^{Is}$.

The semi-discrete PolydG formulation of problem \eqref{eq:coupled_weakform} consists in finding, for all $t\in(0,T]$, the discrete solution $(\bm{u}_h,\bm{w}_h,\varphi_h)(t)\in \bm{V}_h^p\times \bm{V}_h^p\times V_h^a$ such that
\begin{multline}\label{eq::dgsystem}
\hspace{-3mm}
\partial_{tt}\mathcal{M}((\bm{u}_h,\bm{w}_h,\varphi_h), (\bm{v}_h,\bm{z}_h,\psi_h)) +
\mathcal{A}_h((\bm{u}_h,\bm{w}_h,\varphi_h),(\bm{v}_h,\bm{z}_h,\psi_h))+\partial_t\mathcal{B} (\bm{w}_h,{\bm z}_h) \\ 
\hspace{-1mm}+\hspace{-0.5mm}\partial_t[\mathcal{C}_h(\varphi_h,\bm{v}_h\hspace{-0.5mm}+\bm{z}_h)
\hspace{-0.5mm}-\mathcal{C}_h(\bm{u}_h\hspace{-0.5mm}+\bm{w}_h,\psi_h)] \hspace{-0.5mm} = \hspace{-0.5mm}(\bm{f},\bm{v}_h)_{\Omega_p} \hspace{-1mm}+ \hspace{-0.5mm}(\bm{g},\bm{z}_h)_{\Omega_p} \hspace{-1mm}+ \hspace{-0.5mm}(h,\psi_h)_{\Omega_a}
\end{multline}
for all discrete functions $(\bm{v}_h,\bm{z}_h,\psi_h)\in \bm{V}_h^p\times \bm{V}_h^p\times V_h^a$.
As initial conditions we take the $L^2$-orthogonal projections onto $(\bm{V}_h^p\times \bm{V}_h^p\times V_h^a)^2$ of the initial data $(\bm{u}_0,\bm{w}_0,\varphi_0,\bm{u}_1, \bm{w}_1,\varphi_1)$.
For all $\bm{u},\bm{v},\bm{w},\bm{z}\in\bm{V}_h^p$ and  $\varphi,\psi \in V_h^a$, the bilinear forms $\mathcal{A}_h$ and $\mathcal{C}_h$ appearing in \eqref{eq::dgsystem} are given by
\begin{align}
\mathcal{A}_h((\bm{u},\bm{v},\varphi),(\bm{v},\bm{z},\psi))  & = \mathcal{A}_h^e(\bm{u},\bm{v})+
\widetilde{\mathcal{A}}_h^p(\beta\bm{u}+\bm w,\beta\bm{v}+\bm{z}) + \mathcal{A}_h^a(\varphi,\psi), \label{eq:bilinear_Ah} \\
\mathcal{C}_h(\varphi,\bm{v}) &= (\rho_a\varphi,\bm{v}\cdot\bm{n}_p)_{\mathcal{F}_{h}^{I}}, \label{eq::bilineardg2}
\end{align}
with $\mathcal{A}_h^e:{\bm V}_h^p\times{\bm V}_h^p\to\mathbb{R}$ defined as in \eqref{el_B_tilde} and
\begin{equation}
\label{eq:DGbilinearforms}
\begin{aligned}
\widetilde{\mathcal{A}}_h^p(\bm{w},\bm{z}) &=
(m\nabla\cdot\bm{w},\nabla\cdot\bm{z})_{\mathcal{T}_h^p}
-(\llbrace m(\nabla\cdot\bm{w})\rrbrace,\llbracket\bm{z}\rrbracket_{\bm n})_{\mathcal{F}_h^{pi}\cup\mathcal{F}_h^{pD}\cup\mathcal{F}_h^{Is}} \\ 
& \hspace{-4.5mm}  
-(\llbracket\bm{w}\rrbracket_{\bm n},\llbrace m(\nabla\cdot\bm{z})\rrbrace)_{\mathcal{F}_h^{pi}\cup\mathcal{F}_h^{pD}\cup\mathcal{F}_h^{Is}}+(\gamma\llbracket\bm{w}\rrbracket_{\bm{n}},\llbracket\bm{z}\rrbracket_{\bm{n}})_{\mathcal{F}_h^{pi}\cup\mathcal{F}_h^{pD}\cup\mathcal{F}_h^{Is}}, \\
\mathcal{A}_h^a(\varphi,\psi) &= 
(\rho_a\nabla\varphi,\nabla\psi)_{\mathcal{T}_h^a}
-(\llbrace\rho_a\nabla\varphi\rrbrace,\llbracket\psi\rrbracket)_{\mathcal{F}_h^{ai}\cup\mathcal{F}_h^{aD}} \\ 
& \qquad \qquad \qquad -(\llbracket\varphi\rrbracket, \llbrace\rho_a\nabla\psi\rrbrace)_{\mathcal{F}_h^{ai}\cup\mathcal{F}_h^{aD}}+(\chi\llbracket\varphi\rrbracket,\llbracket\psi\rrbracket)_{\mathcal{F}_h^{ai}\cup\mathcal{F}_h^{aD}}.
\end{aligned}
\end{equation}
Notice that the bilinear form $\widetilde{\mathcal{A}}_h^p$ is different from $\mathcal{A}_h^p$ defined in \eqref{eq:Ahp}. Indeed, the definition of $\widetilde{\mathcal{A}}_h^p$ in \eqref{eq:DGbilinearforms} also takes into account the essential condition $\bm z\cdot\bm n_p = 0$ on $\Gamma_I^s$ embedded in the definition of the functional space $\bm W_\tau$.
The stabilization function $\chi\in L^\infty(\mathcal{F}_h^a)$ is defined such that
\begin{equation}\label{def:penalty_acoustic}
  \chi = \rho_0 
\begin{cases}
\underset{\kappa\in\{\kappa_1,\kappa_2\} } \max \left((\rho_a)_{\mid\kappa}\, {r}^2_\kappa h_\kappa^{-1}\right), & F \in \mathcal{F}_h^{ai}, \, F \subset \partial \kappa_1 \cap \partial \kappa_2, \\
(\rho_a)_{\mid\kappa} {r}^2_\kappa h_\kappa^{-1}, &  F\in\mathcal{F}_h^{aD}, \, F\subset\partial\kappa\cap \Gamma_{aD},
\end{cases}
\end{equation}
with $\rho_0>0$ being a user-dependent parameter.

Denoting by ($U_h$, $W_h$,$\Phi_h$) the vector of the coefficients of $(\bm{u}_h,\bm{w}_h,\varphi_h)$ in the chosen basis for $\bm V_h^p\times \bm V_h^p \times V_h^a$, the algebraic form of problem \eqref{eq::dgsystem} reads:
\begin{multline}\label{eq::algebraic_coupled}
\left[ \begin{matrix} 
\bm{M}^p_\rho & \bm{M}^p_{\rho_f} & 0 \\
\bm{M}^p_{\rho_f} & \bm{M}^p_{\rho_w} & 0 \\
0 & 0 & \bm{M}^a_{\rho_a c^{-2}}
\end{matrix} \right]  
\left[ \begin{matrix} 
\ddot{U}_h \\
\ddot{W}_h \\ 
\ddot{\Phi}_h
\end{matrix} \right] + 
\left[ \begin{matrix} 
0 & 0 & \bm{C} \\
0 & \bm{B} & \bm{C} \\
-\bm{C} & -\bm{C} & 0
\end{matrix} \right]  
\left[ \begin{matrix} 
\dot{U}_h \\
\dot{W}_h \\ 
\dot{\Phi}_h
\end{matrix} \right] \\
+\left[ \begin{matrix} 
\bm{A}^e + \widetilde{\bm{A}}^p_{\beta^2} &  \widetilde{\bm{A}}^p_{\beta} & 0 \\
\widetilde{\bm{A}}^p_{\beta} & \widetilde{\bm{A}}^p & 0 \\ 
0 & 0 & \bm{A}^a
\end{matrix} \right] 
\left[ \begin{matrix} 
{U}_h \\
{W}_h \\ 
{\Phi}_h
\end{matrix} \right] = 
\left[ \begin{matrix} 
F_h \\
G_h \\ 
H_h
\end{matrix} \right],
\end{multline}
with initial conditions $(U_h,W_h,\Phi_h)(0)=(U_0,W_0,\Phi_0)$ and $(\dot{U}_h,\dot{W}_h,\dot{\Phi}_h)(0)=(U_1,W_1,\Phi_1)$. With the notation introduced in Section \ref{sec:time_int}, problem \eqref{eq::algebraic_coupled} can be rewritten in the form of \eqref{eq:2ndorder_time} by setting $X_h = [U_h, W_h, {\Phi}_h]^{\rm T}$, $S_h =[F_h, G_h, H_h]^{\rm T}$, and 
$$
\bm{M}_h \hspace{-0.5mm}=\hspace{-0.5mm} \left[ \begin{matrix} 
\bm{M}^p_\rho & \bm{M}^p_{\rho_f} & 0 \\
\bm{M}^p_{\rho_f} & \bm{M}^p_{\rho_w} & 0 \\
0 & 0 & \bm{M}^a_{\frac{\rho_a}{c^{2}}}
\end{matrix} \right]\hspace{-1mm}, \;\
\bm{D}_h \hspace{-0.5mm}=\hspace{-0.5mm} \left[ \begin{matrix} 
0 & 0 & \bm{C} \\
0 & \bm{B} & \bm{C} \\
-\bm{C} & -\bm{C} & 0
\end{matrix} \right]\hspace{-1mm}, \;\
\bm{A}_h \hspace{-0.5mm}=\hspace{-0.5mm} \left[ \begin{matrix} 
\bm{A}^e + \widetilde{\bm{A}}^p_{\beta^2} &  \widetilde{\bm{A}}^p_{\beta} & 0 \\
\widetilde{\bm{A}}^p_{\beta} & \widetilde{\bm{A}}^p & 0 \\ 
0 & 0 & \bm{A}^a
\end{matrix} \right]\hspace{-1mm}.
$$

\subsection{Stability and convergence results}

In this section, we present the main stability and convergence results proved in \cite{Antonietti.ea:21}. First, we introduce the energy norm defined such that, for all $(\bm u, \bm w,\varphi)\in C^1([0,T];\bm{V}_h^p\times\bm{V}_h^p\times V_h^a)$, 
\begin{multline}\label{eq:porelac_norm}
    \|(\bm u, \bm w, \varphi)(t)\|^2_{\mathbb{E}} = 
    \|(\bm u, \bm w)(t)\|^2_{\mathcal{E}} +
    \|\rho_a^{\frac12}c^{-1}\partial_t\varphi(t) \|^2_{\Omega_a} 
    + \|\varphi(t)\|_{\textrm{DG},a}^2 \\
    + \|\gamma^{\frac12}\bm w\cdot\bm n\|_{\mathcal{F}_h^{Is}}^2 +
    \int_0^t \|\zeta_\tau^{\frac12}\partial_t(\bm w\cdot\bm n)\|_{\mathcal{F}_h^{Io}}^2\,{\rm d}s,
\end{multline}
with $\|\cdot\|_{\mathcal{E}}$ defined in \eqref{eq:porel_norm} and $\|\cdot\|_{\textrm{DG},a}:\bm V_h^a\to\mathbb{R}^+$ given by
\begin{equation}\label{acu_dg_norm}
\| \varphi \|_{\textrm{DG},a}^2  =  
\norm{\rho_a^{\frac12} \nabla\varphi}_{\mathcal{T}_h^a}^2
+\norm{\chi^{\frac{1}{2}} \,\llbracket\varphi\rrbracket }_{\mathcal{F}_h^{aI} \cup \mathcal{F}_h^{aD}}^2 \qquad\forall \varphi \in \bm V_h^a \oplus H^1_0(\Omega_a).
\end{equation}

The stability of the semi-discrete PolydG problem \eqref{eq::dgsystem} is a consequence of Proposition \ref{prop:stab_coupled} below, which also implies that the formulation is dissipative. Indeed, in the case of null external source terms, it follows from estimate \eqref{eq:stab_porelac} that $\norm{(\bm u_h,\bm w_h,\varphi_h)(t)}_{\mathbb{E}}\lesssim \norm{(\bm u_h,\bm w_h,\varphi_h)(0)}_{\mathbb{E}}$ for any $t>0$.
The proof of the following result is based on taking $(\bm{v}_h,\bm{z}_h,\psi_h)=(\partial_t{\bm{u}}_h,\partial_t{\bm{w}}_h,\partial_t{\varphi}_h) \in \bm{V}_h^p\times \bm{V}_h^p\times V_h^a$ in \eqref{eq::dgsystem}, using the skew-symmetry of the coupling terms, and then reasoning as in Proposition \ref{prop:stability_porel} (see \cite[Theorem 3.4]{Antonietti.ea:21} for the details). 
\begin{prop}\label{prop:stab_coupled}
For sufficiently large penalty parameters $\sigma_0, m_0, \rho_0$ and for any $t\in(0,T]$, the solution $(\bm u_h,\bm w_h,\varphi_h)(t)\in \bm{V}_h^p\times \bm{V}_h^p\times V_h^a$ of \eqref{eq::dgsystem} satisfies
\begin{equation}\label{eq:stab_porelac}
\norm{(\bm u_h,\bm w_h,\varphi_h)(t)}_{\mathbb{E}}\lesssim\norm{(\bm u_h,\bm w_h,\varphi_h)(0)}_{\mathbb{E}} 
+\int_0^t\norm{\bm f(s)}_{\Omega_p}^2 +\norm{\bm g(s)}_{\Omega_p}^2 +\norm{h(s)}_{\Omega_a}^2 \rm{d}s,
\end{equation}
with hidden constant depending on time $t$ and on the material properties, but independent of the interface parameter $\tau$.
\end{prop}

In what follows, we report the main result concerning the error analysis of the PolydG discretization \eqref{eq::dgsystem}. To infer the error estimate of Theorem \ref{thm:poelac_error} below, an additional assumption on the interface permeability $\tau$ is required.
\begin{assumption}\label{ass::tau_bnd}
For each $F\in\mathcal{F}_h^{Io}$ and $\kappa\in\mathcal{T}_h^p$ such that $F\subset\partial\kappa\cap\Gamma_I^o$, it holds $(\zeta_\tau)_{\mid F}=(\frac{1-\tau}\tau)_{\mid F} \lesssim \frac{p_{p,\kappa}^2}{h_{\kappa}}$, with hidden constant independent of $\tau$.
\end{assumption}
\noindent
We remark that the previous assumption is used only for establishing the error estimate below but, according to our observation, it is not needed in practical applications.
We refer the reader to \cite[Theorem 4.3]{Antonietti.ea:21} for the detailed proof of the following result.
\begin{theorem} \label{thm:poelac_error}
Let Assumption \ref{ass::regular}, Assumption \ref{ass::3}, and Assumption \ref{ass::tau_bnd} be satisfied and assume that the solution $(\bm u,\bm w,\varphi)$ of the weak formulation \eqref{eq:porel_weak} is sufficiently regular. For any time $t \in [0,T]$, let $(\bm u_h,\bm w_h,\varphi_h)(t)\in \bm{V}_h^p\times \bm{V}_h^p\times V_h^a$ be the PolydG solution of problem~\eqref{eq::dgsystem} obtained with sufficiently large penalization parameters $\sigma_0,m_0$ and $\rho_0$. Then, for any time $t \in (0,T]$,  the discretization error $\bm E(t)=(\bm u - \bm u_h, \bm w - \bm w_h, \varphi-\varphi_h)(t)$ satisfies
\begin{equation}\label{eq:poelac_error-estimate}
\begin{aligned}
\| \bm E (t)\|_{\mathbb{E}} 
&\lesssim \sum_{\kappa \in \mathcal{T}_h^p} \frac{h_\kappa^{s_\kappa-1}}{p_\kappa^{m_\kappa-3/2}} \left(
\mathcal{I}_{m_\kappa}^{\mathcal{T}_\sharp}(\bm u,\bm w)(t) + \int_{0}^{t} \mathcal{I}_{m_\kappa}^{\mathcal{T}_\sharp}(\partial_t{\bm u}, \partial_t\bm w)(s) \, ds\right) \\
&+ \sum_{\kappa \in \mathcal{T}_h^a} \frac{h_\kappa^{q_\kappa-1}}{r_\kappa^{l_\kappa-3/2}} \left(
\mathcal{I}_{l_\kappa}^{\mathcal{T}_\sharp}(\varphi)(t) + \int_{0}^{t} \mathcal{I}_{l_\kappa}^{\mathcal{T}_\sharp}(\partial_t\varphi)(s) \, ds\right),
\end{aligned}
\end{equation}
where
\begin{equation*}
  \begin{aligned}
    \mathcal{I}_{m_\kappa}^{\mathcal{T}_\sharp}(\bm u, \bm w) &= 
    \| \widetilde{\mathcal{E}} \bm u \|_{\bm H^{m_\kappa}(\mathcal{T}_\sharp)} \hspace{-0.6mm}+ 
    \| \widetilde{\mathcal{E}} \bm w \|_{\bm H^{m_\kappa}(\mathcal{T}_\sharp)} \hspace{-0.6mm}+
    \| \widetilde{\mathcal{E}} \partial_t{\bm u} \|_{\bm H^{m_\kappa}(\mathcal{T}_\sharp)} \hspace{-0.6mm}+
    \| \widetilde{\mathcal{E}} \partial_t{\bm w} \|_{\bm H^{m_\kappa}(\mathcal{T}_\sharp)},\\
    \mathcal{I}_{l_\kappa}^{\mathcal{T}_\sharp}(\varphi) &= 
    \| \widetilde{\mathcal{E}} \varphi \|_{ H^{l_\kappa}(\mathcal{T}_\sharp)} + 
    \| \widetilde{\mathcal{E}} \partial_t{\varphi} \|_{ H^{l_\kappa}(\mathcal{T}_\sharp)},
  \end{aligned}
\end{equation*}
with $s_\kappa = \min(p_\kappa +1, m_\kappa)$ and $q_\kappa = \min(r_\kappa +1, l_\kappa)$ for all $\kappa \in \mathcal{T}_h$. The hidden constant depends on time $t$, the material properties, and the shape-regularity of the covering $\mathcal{T}_\sharp$, but is independent of the discretization parameters and of $\tau$.
\end{theorem}

\subsection{Verification test}\label{ver_test_poroac}

\begin{figure}
    \centering
    \includegraphics[width=0.5\textwidth]{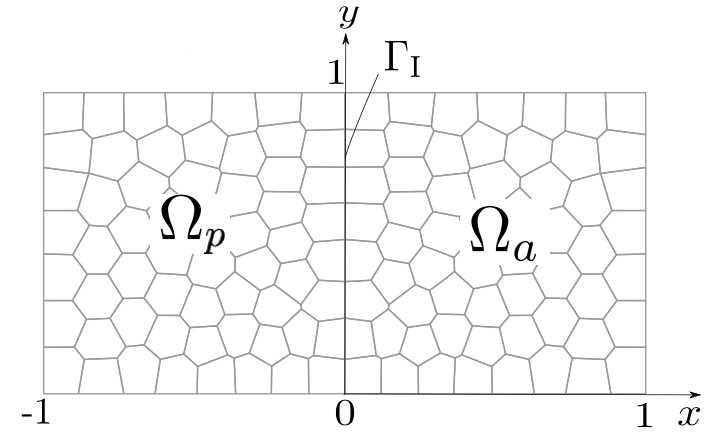}
    \caption{Test case of Section~\ref{ver_test_poroac}. Polygonal mesh with $N_{el} =100$ elements.}
    \label{fig::mesh}
\end{figure}
As a verification test case, we study the poro-elasto-acoustic problem coupling \eqref{eq:model_porel} and \eqref{eq::acousticeq} with the interface conditions \eqref{eq:interface} in the domain $\Omega=\Omega_p\cup\Omega_a=(-1,1)\times(0,1)$.  
We consider a sequence of polygonal meshes as the one shown in Figure \ref{fig::mesh}, the physical parameters listed in Figure \ref{poroelastic-mesh-param} (right) and  $c=\rho_a = 1$.  As exact solution we consider 
\eqref{solution_poroelastic} in $\Omega_p$ and 
\begin{equation*}
\varphi(x,y;t)=
 x^2 \sin(\pi x)\sin(\pi y) \sin(\sqrt{2}\pi t),    
\end{equation*}
in $\Omega_a$ in order to have a null pressure in the whole poroelastic domain. Dirichlet and initial conditions are set accordingly. We remark that with this choice the interface coupling conditions are null on $\Gamma_I$. For the following test cases we consider $\tau = 1$ (open pores) at the interface, however similar results can be obtained with $\tau \in [0,1)$, cf. \cite{Antonietti.ea:21}.
We fix the $T=0.25$ and consider a time step $\Delta t=10^{-4}$ for the Newmark-$\beta$ scheme,  $\gamma_N=1/2$ and $\beta_N=1/4$, cf. \eqref{eq::newmark_taylor}.
Penalty parameters $\sigma_0$ and $m_0$ in $\Omega_p$ as well as $\rho_0 \in \Omega_a$ are set equal to 10, cf. \eqref{el_penalization_parameter}, \eqref{pen_m_parameter}  and \eqref{def:penalty_acoustic}, respectively.   

\begin{figure}
\begin{minipage}{0.6\textwidth}
    \begin{tabular}{|c|c|c|c|}
    \hline 
    h \textbackslash \, p  &  2 & 3 & 4 \\
    \hline
     0.35   & 5.5007e-1 &  4.4095e-3 &  7.6593e-3\\
     0.25   & 2.0885e-1 &  1-6164e-3 &  1.7884e-3 \\
     0.18   & 1.3225e-1 &  6.5993e-4 &  5.1521e-4 \\
     0.13   & 6.9271e-2 &  2.5537e-4 &  1.2756e-4 \\
     \hline 
     rate & 1.98 & 2.91 & 4.29 \\
     \hline
    \end{tabular} 
\end{minipage}
\hspace{2mm}
\begin{minipage}{0.4\textwidth}
\includegraphics[width=0.9\textwidth]{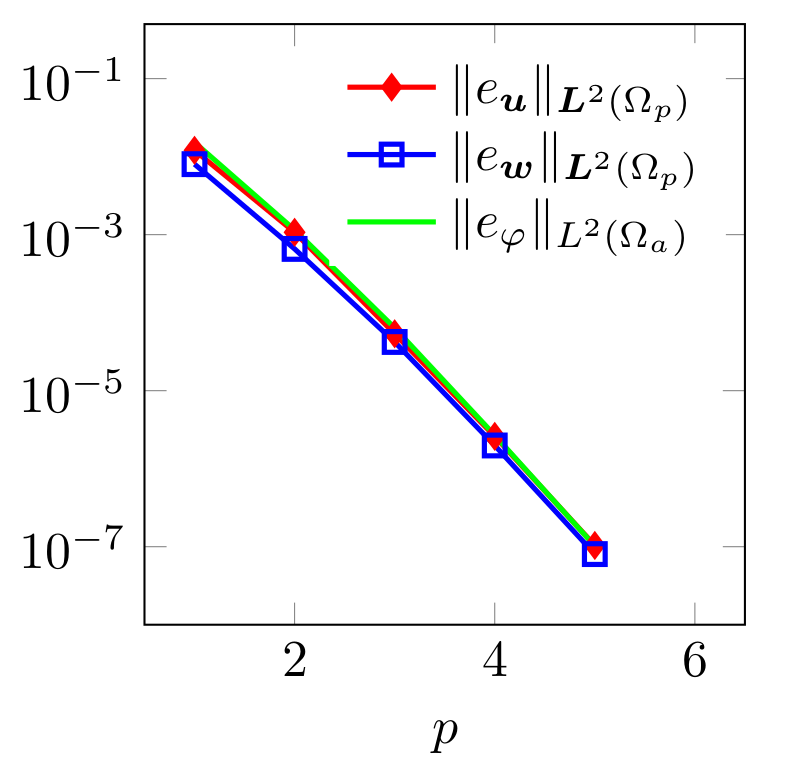}
\end{minipage}
    \caption{Test case of Section~\ref{ver_test_poroac}. Computed energy error as a function of the mesh size $h$ for polynomial degree $p=2,3,4$. The rate of convergence is reported in the last row, cf. \eqref{eq:poelac_error-estimate} (left). Computed $\bm L^2$-errors  $\|e_{\bm u}\|_{\bm L^2(\Omega_p)} = \|\bm u-\bm u_h\|_{\bm L^2(\Omega_p)}$, $\|e_{\bm w}\|_{\bm L^2(\Omega_p)} = \|\bm w-\bm w_h\|_{\bm L^2(\Omega_p)}$ and $\|e_{\varphi}\|_{L^2(\Omega_a)} = \|\varphi-\varphi_h\|_{ L^2(\Omega_a)}$ as a function of the polynomial degree $p$ in a semilogarithmic scale with fixed the  number of polygonal elements as $N_{el} = 100$ (right).}
    \label{fig:convergence_poro_acoustic}
\end{figure}

Finally, in Figure \ref{fig:convergence_poro_acoustic} (left) we report the computed energy errors as a function of the  the mesh-size $h$, for the $p=2,3,4$. Consistently with \eqref{eq:poelac_error-estimate} the errors decays proportionally to $h^{p}$.
In Figure \ref{fig:convergence_poro_acoustic} (right) we plot in a semilog-scale the computed $L^2$-norms of the error fixing a computational mesh of $N_{el}=100$ polygons and varying the polynomial degree $p=1,2,\ldots, 5$. An exponential decay of the error is clearly attained.


\section{Examples of physical interest}\label{sec:NumResults}

\subsection{Two layered media}\label{layered_media}
In this section we consider a wave propagation problem in heterogeneous media taken from \cite{Tromp_Morency_2008}. The aim of this test is to show how different assumptions on the model can determine and change the behavior of the wave propagation. 

The domain of interest is $\Omega = (0,4.8)^2~$ $\rm{km}^2$ and consists of two layers as depicted in Figure~\ref{fig:layerd_media}. In the first case (a) the layers are perfectly elastic, cf. Table~\ref{tab::table_elastic}, while in the second case (b) the layers are assumed to be poro-elastic, cf. Table~\ref{tab::table_poroelastic}. A point-wise source $\bm f$, cf. \eqref{point-force}, acting in the $y-$ direction is located in the upper part of the domain at point $\bm x = (2.4,2.7)~\rm{km}$. The  time evolution of the latter is given by a Ricker-wavelet \eqref{eq:ricker}  with amplitude $A_0=1$~m, time-shift $t_0=0.3$~s and peak-frequency $f_p=5$ Hz. For both models (a) and (b) we use a polygonal mesh with characteristic size $h=10^{-2}$ and a polynomial degree $p=3$. We set homogeneous Dirichlet conditions on the boundary and use null initial conditions. To integrate in time model (a) we chose the leap-frog scheme while for model (b) the Newmark-$\beta$ scheme with parameters $\beta_N$ and $\gamma_N$ as in the previous section.  We fix the final time $T=1$ s and chose $\Delta t = 10^{-3}$ s.    
\begin{figure}
    \centering
    \includegraphics[width=0.5\textwidth]{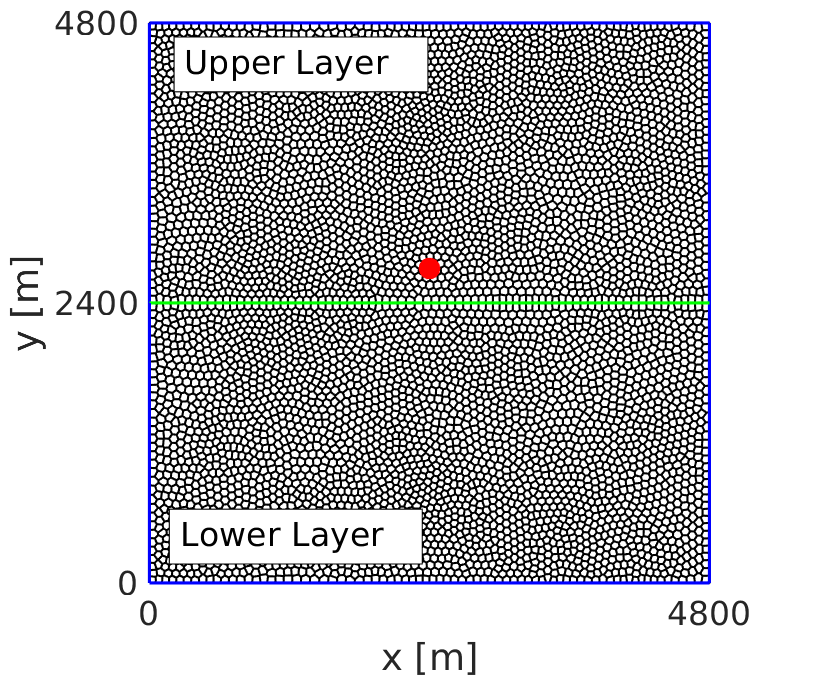}
    \caption{Test case of Section \ref{layered_media}. Computational domain: the location of the point-source force is superimposed in red. }
    \label{fig:layerd_media}
\end{figure}

\begin{table}[htbp]
\centering
\begin{tabular}{lllllll}
& & & Lower Layer & Upper Layer & & \\
\cline{1-7}
 & Solid density      & $\rho$    & 2650 & 2200                 & $\rm kg/m^3$    &  \\
& Shear modulus      & $\mu$       & 1.5038 $\cdot 10^9$ & 4.3738 $\cdot 10^9$     & $\rm Pa$        &  \\
 & Lam\'e coefficient   & $\lambda$ & 1.8121$\cdot 10^9$ & 7.2073$\cdot 10^9$    & $\rm Pa$        &  \\
 & Damping coefficient   & $\zeta$ & 0 & 0    & $\rm s^{-1}$        &  \\
  \cline{1-7}
\end{tabular}
\caption{Test case of Section \ref{layered_media}. Physical parameters for the elastic medium.}
\label{tab::table_elastic}
\end{table}
\begin{table}[htbp]
\centering
\begin{tabular}{lllllll}
& & & Lower Layer & Upper Layer & & \\
\cline{1-7}
\textbf{Fluid}  & Fluid density      & $\rho_f$    & 750 & 950                 & $\rm kg/m^3$    &  \\
& Dynamic viscosity  & $\eta$      & 0  & 0                   & $\rm Pa\cdot s$ &  \\ \cline{1-7}
\textbf{Grain}  & Solid density      & $\rho_s$    & 2650 & 2200                 & $\rm kg/m^3$    &  \\
& Shear modulus      & $\mu$       & 1.5038 $\cdot 10^9$ & 4.3738 $\cdot 10^9$     & $\rm Pa$        &  \\ \cline{1-7}
\textbf{Matrix} & Porosity           & $\phi$      & 0.2 & 0.4                 &             &  \\
& Tortuosity         & $a$         & 2 & 2                  &             &  \\
& Permeability       & $k$         & $1\cdot 10^{-12}$ & $1\cdot 10^{-12}$ & $\rm m^2$       &  \\
& Lam\'e coefficient   & $\lambda$ & 1.8121$\cdot 10^9$ & 7.2073$\cdot 10^9$    & $\rm Pa$        &  \\
& Biot's coefficient & $m$         & 7.2642$\cdot 10^9$ & 6.8386$\cdot 10^9$   & $\rm Pa$        &  \\
& Biot's coefficient & $\beta$     & 0.9405  & 0.0290                 &             &  \\  \cline{1-7}
\end{tabular}
\caption{Test case of Section \ref{layered_media}. Physical parameters for the poro-elastic medium.}
\label{tab::table_poroelastic}
\end{table}

In Figure~\ref{fig:layerd-model} we report selected snapshots of the computed magnitude of the velocity field $\mid \partial_t \bm{u}_h(t) \mid$ for models (a) and (b). As expected, the propagation of the wave in the elastic domain is regular and refraction phenomena are not very evident (due to a low contrast between the wave speeds). On the contrary, when porous media are accounted for, the refraction effects are more pronounced. This is in agreement with the findings in \cite{Tromp_Morency_2008}.

\begin{figure}[htbp]
\centering
\includegraphics[width=0.3\textwidth]{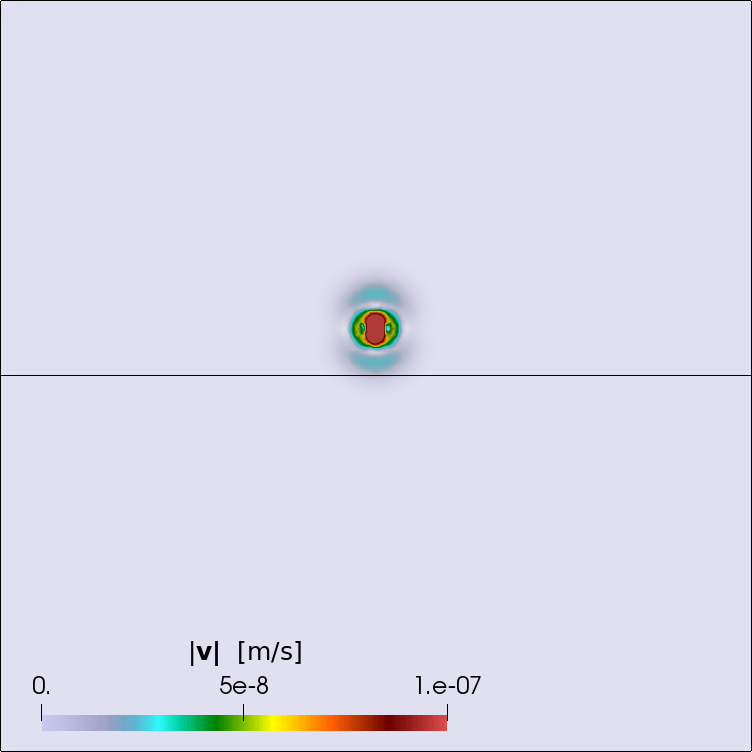}
\includegraphics[width=0.3\textwidth]{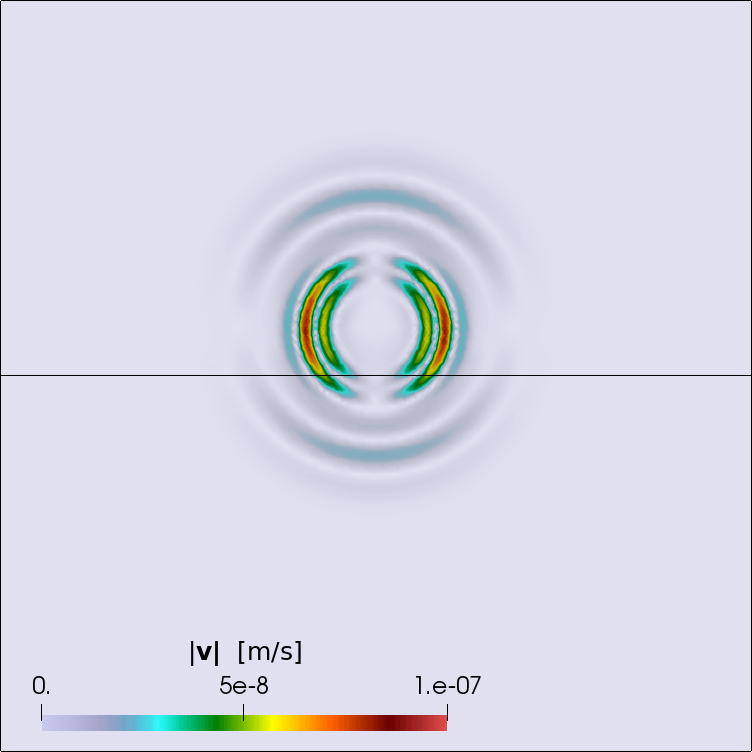}
\includegraphics[width=0.3\textwidth]{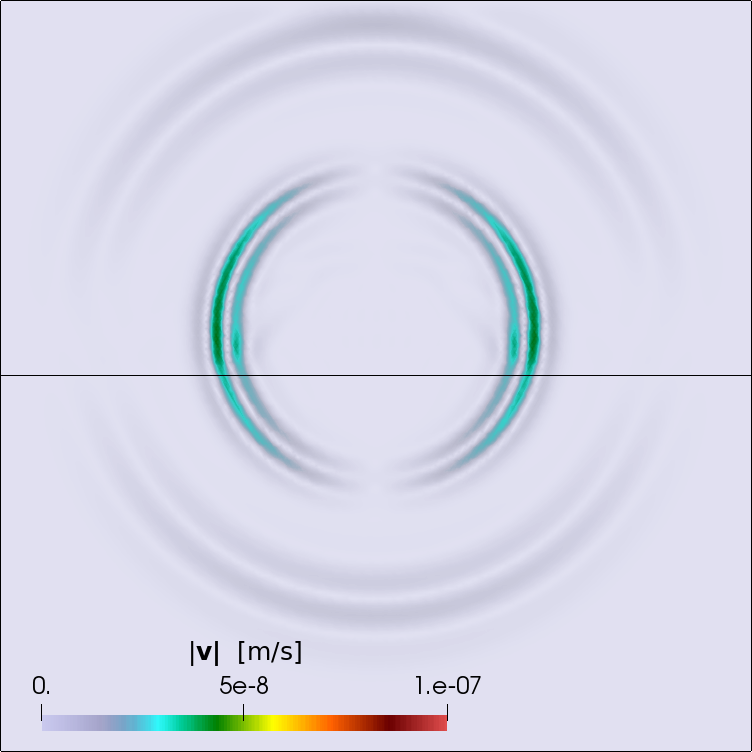}
\includegraphics[width=0.3\textwidth]{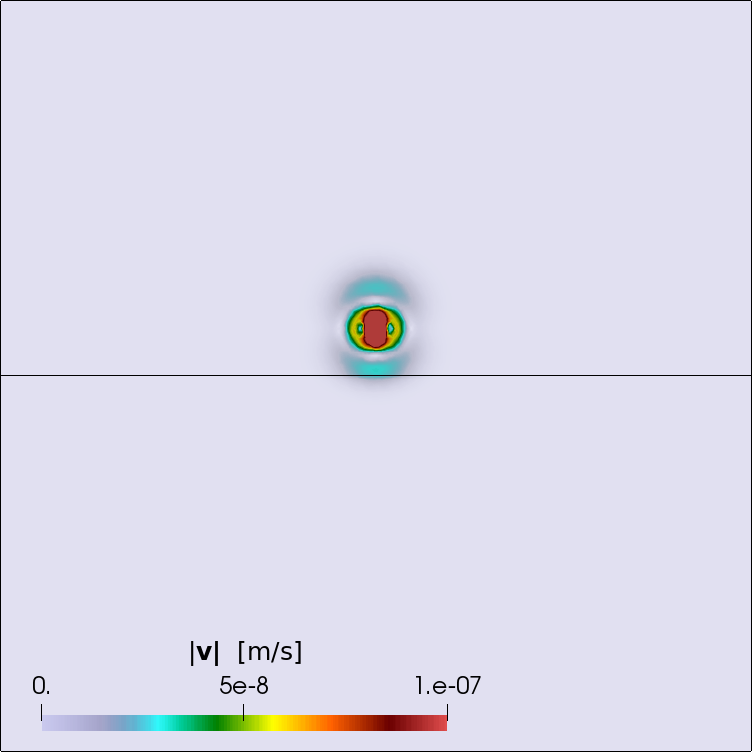}
\includegraphics[width=0.3\textwidth]{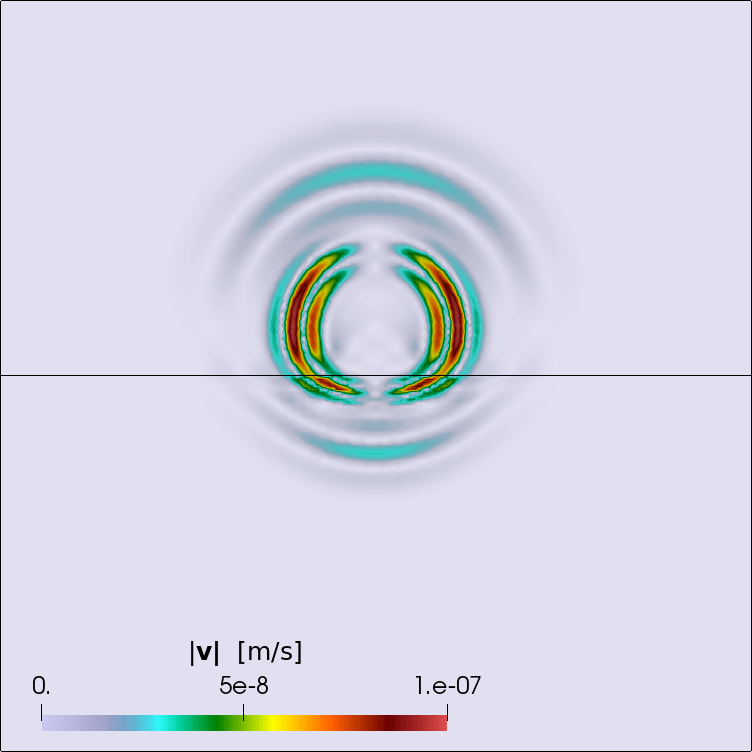}
\includegraphics[width=0.3\textwidth]{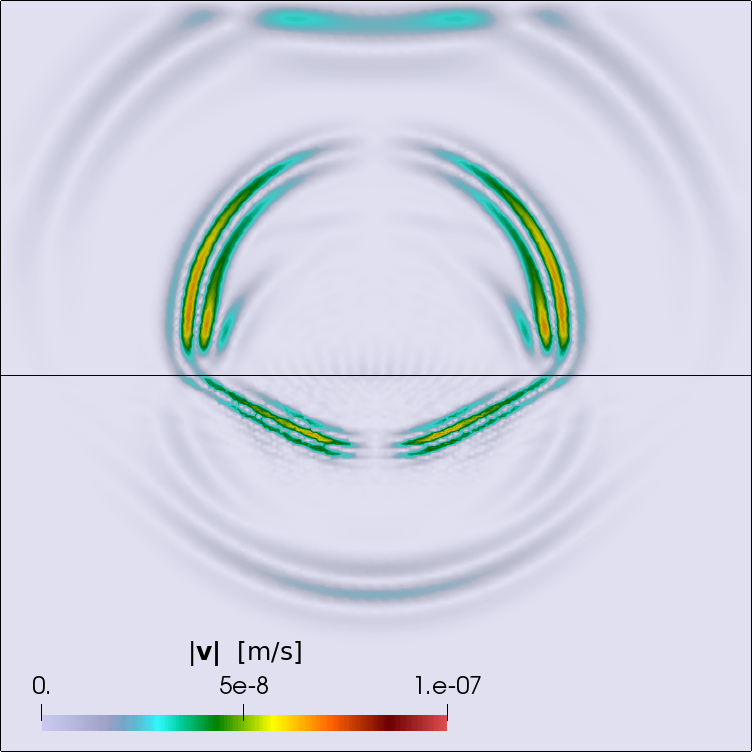}
\caption{Test case of Section \ref{layered_media}.  Computed velocity field $\mid \partial_t\bm{u}_h(t) \mid$ at the time instants $t=0.3$~s (left), $t=0.6$~s (center) and $t=1$~s (right) for elastic model (a) (top) and poro-elastic model (b) (bottom).}
\label{fig:layerd-model}
\end{figure}

\subsection{Wave propagation in layered poro-elastic-acoustic media}\label{layered_poro_ac}

As a final test cases we consider the domain reproduced in Figure~\ref{fig:emilia_mesh}
where an acoustic layer is in contact with a heterogeneous poro-elastic body.
\begin{figure}
    \centering
    \includegraphics[width=1\textwidth]{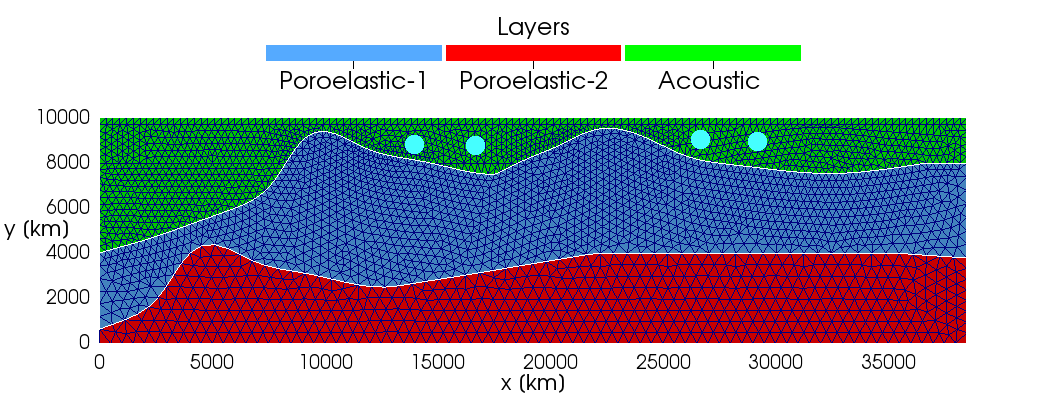}
    \caption{Test case of Section \ref{layered_poro_ac}. Computational domain. Location of the acoustic sources are also superimposed.}
    \label{fig:emilia_mesh}
\end{figure}

For the acoustic domain we set $\rho_a=1500~\rm{[kg/m^3]}$ and $c=1000~\rm{[m/s]}$. Physical parameters for the poro-elastic domain are chosen as in Table~\ref{tab::table_poroelastic} where, for this case, the property of the former ``Lower Layer" are assigned to the first poro-elastic subdomain, while those of the former ``Upper Layer" to the second poro-elastic subdomain, cf. Figure~\ref{fig:emilia_mesh}.
In this numerical example we chose the dynamic viscosity $\eta$ equal to $0.001$.
Boundary and initial conditions have been set equal to zero both for the poroelastic and the acoustic domain. Forcing terms are null in $\Omega_p$, while in $\Omega_a$ we consider a force of the form $h=r(x,y)q(t)$,
where $q$ is a Ricker wavelet of the form \eqref{eq:ricker} with $A_0 = 1~\rm{[Hz~m^3]}$, $\beta_p = 39.4784~\rm{[Hz^2]}$ and $t_0=0.75$~s. The function  $r(x,y)$ is  defined as $ r(x,y)= 1$,  if $(x,y) \in \bigcup_{i=1}^4 B({\bm x}_i,R)$, while $r(x,y)=0$, otherwise, where $B({\bm x}_i,R)$ is the circle centered in ${\bm x}_i$ and with radius $R$. Here, we set ${\bm x}_1 = (13097,8868)$ m, ${\bm x}_2 = (16673,8868)$ m, ${\bm x}_3 = (27079,8868)$ m, ${\bf x}_4 = (29324,8868)$ m and $R=100$ m. 
Notice that, the support of the function $r(x,y)$ has been reported in Figure~\ref{fig:emilia_mesh}, superimposed with a sample of the computational mesh employed.

Simulations have been carried out by considering: a mesh consisting in $N=6356$ triangles, subdivided into $N_a= 2380$ and $N_p=3976$ triangles for the acoustic and poroelastic domain, respectively; a Newmark scheme with time step $\Delta t=10^{-2}\ \rm s$ and $\gamma_N=1/2$ and $\beta_N=1/4$ in a time interval $[0,4]\ \rm s$; a polynomial degree $p_{\kappa}=r_{\kappa} = p = 4$. 
In Figure~\ref{fig:emilia_test}, we show the computed pressure $p_h$ 
considering the interface permeability $\tau=1$.  The latter value models an \textit{open} pores condition at the interface, cf. \eqref{eq:interface}. Remark that $p_h = \rho_a \dot{\varphi}_h$ in the acoustic domain while $p_h =  -m(\beta \nabla\cdot\bm{u}_h+\nabla\cdot \bm{w}_h)$ in the poro-elastic one. As one can see, the pressure wave correctly propagates from the acoustic domain to the poro-elastic one: the continuity at the interface boundary can be appreciated.  Finally, we note how the second porous layer (sound absorbing material) produces a damping of the pressure field.

\begin{figure}
\centering
\includegraphics[width=1\textwidth]{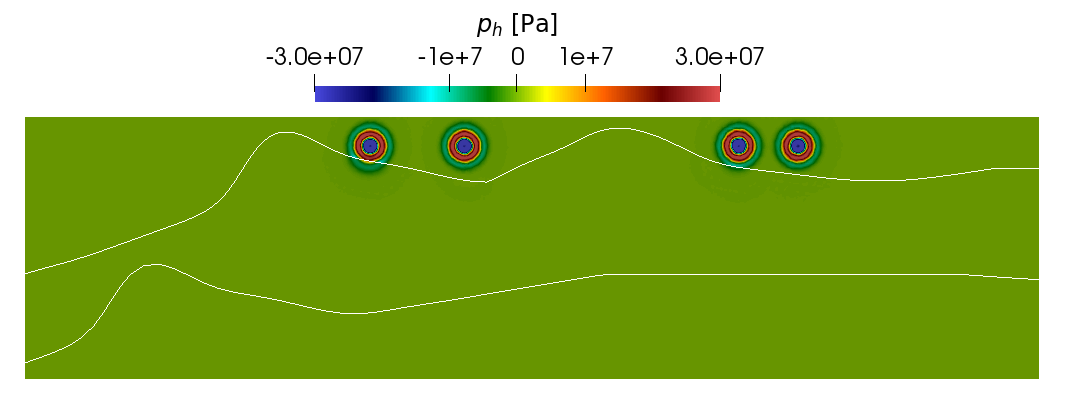}
\includegraphics[width=1\textwidth]{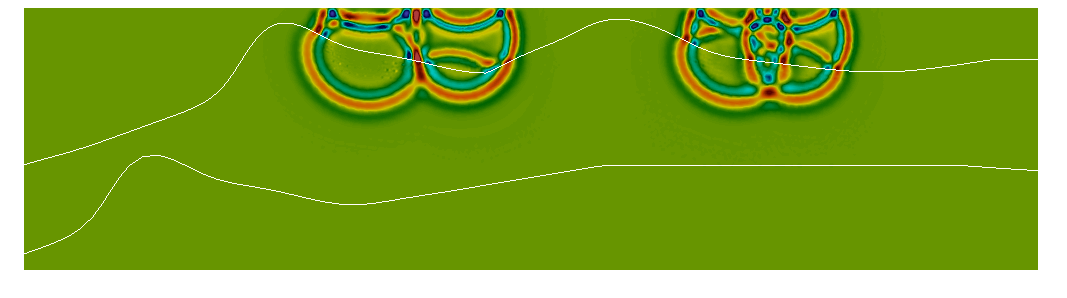}
\includegraphics[width=1\textwidth]{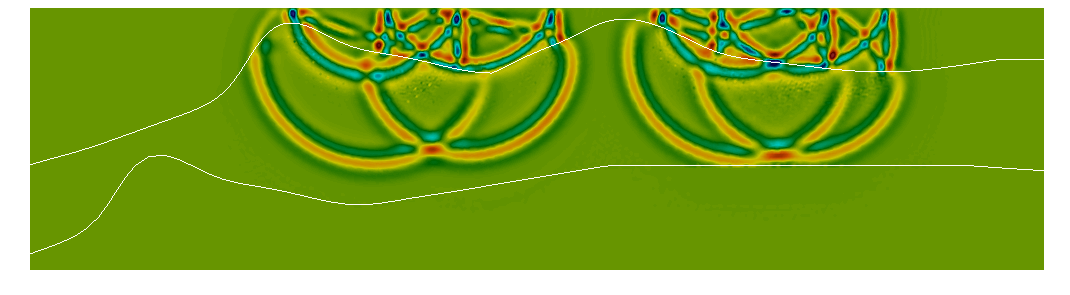}
\includegraphics[width=1\textwidth]{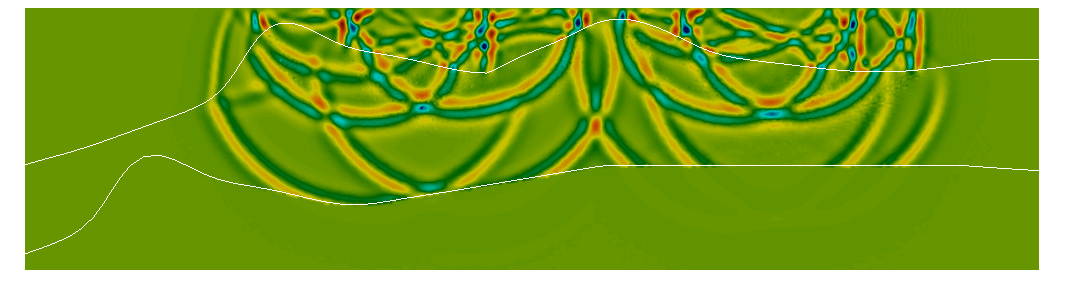}
\caption[Test case of Section \ref{layered_poro_ac}. Pressure solution in the poro-elastic-acoustic domain at four time instants]{Test case of Section \ref{layered_poro_ac}. Computed pressure $p_h$ in the poro-elastic-acoustic domain at four time instants (from up to down $t=1,\ 2,\ 3, \ 3.8 \rm s$), with $\Delta t = 10^{-2}\ \rm s$.}
\label{fig:emilia_test}
\end{figure}


\section{Conclusions}\label{sec:conclusions}

In this work we have presented a review of the development of PolyDG methods for  multiphysics wave propagation phenomena in elastic, poro-elastic and poro-elasto-acoustic media. 

After having recalled the theoretical background of the analysis of PolyDG methods we analysed the well-posedness and stability of different numerical formulations and proved $hp$-version a priori error estimates for the semi-discrete scheme.  
Time integration of the latter is obtained based on employing the leap-frog or the a Newmark methods. 
Numerical experiments have been designed not only to verify the theoretical error bounds but also to demonstrate the flexibility in the process of mesh design offered by polytopic elements. In this respect, numerical tests of physical interest have been also discussed. 

To conclude, PolyDG methods allow a robust and flexible numerical discretization that can be successfully applied to wave propagation problems.
Future developments in this direction include the study of multi-physics problems such as fluid-structure (with poro-elastic or thermo-elastic structure) interaction problems (we refer, e.g., to \cite{AntoniettiVeraniVergaraZonca_2019,Zonca_et_al_2021} for preliminary results) as well as the exploitation of  algorithms to design agglomeration-based multigrid methods and preconditioners for the efficient iterative solution of the (linear) system of equations stemming from PolyDG discretizations (see \cite{AntoniettiHoustonHuSartiVerani_2017,BottiColomboBassi_2017,AntoniettiPennesi_2019,AntonHoustonPennSuli_2020,BottiColomboCrivelliniFranciolini_2019} for seminal results). 


\section*{Declarations}
\textbf{Funding.} This work has received funding from the European Union's Horizon 2020 research and innovation programme under
the Marie Sk\l{}odowska-Curie grant agreement no. 896616 (project PDGeoFF: Polyhedral Discretisation Methods for
Geomechanical Simulation of Faults and Fractures in Poroelastic Media). The authors are members of the INdAM Research Group GNCS and this work is partially funded by INdAM-GNCS. Paola F. Antonietti has been partially funded by the research project PRIN n. 201744KLJL funded by MIUR.
\textbf{Conflict of interest.} The authors have no conflicts of interest to declare that are relevant to the content of this article. \textbf{Availability of data and material.} The datasets generated
during the current study are available from Ilario Mazzieri upon reasonable request. \textbf{Authors' contributions.} All authors contributed to the study conception and design.
The first draft of the manuscript was written by Ilario Mazzieri and Michele Botti. 
All authors commented on previous versions of the manuscript and approved the final one.
\textbf{Ethics approval.} Not applicable. 


\raggedright
{\small\bibliography{references}{}}

\end{document}